\documentclass{article}
\usepackage{amsthm}
\theoremstyle{definition}
\newtheorem{definition}{Definition}
\newtheorem{remark}{Remark}

\theoremstyle{theorem}
\newtheorem{proposition}{Proposition}
\newtheorem{lemma}{Lemma}
\newtheorem{theorem}{Theorem}
\newtheorem{corollary}{Corollary}
\theoremstyle{definition}

\newtheorem{fact}{Fact}
\pdfoutput=1

\newcommand{\cf}{cf.\ }
\newcommand{\st}{s.t.\ }

\usepackage{xspace}
\newcommand{\etc}{etc.\xspace}
\newcommand{\eg}{e.g.,\xspace}
\newcommand{\ie}{i.e.,\xspace}

\usepackage{array}
\usepackage[pdftex,dvipsnames]{color}
\usepackage{ifthen,calc,url}

\usepackage{amstext,amsmath,amssymb,stmaryrd,mathrsfs}
\usepackage[all,cmtip]{xy}
\usepackage{rotating}

\newcommand{\defeq}{\mathrel{:=}}
\newcommand{\defiff}{\mbox{:iff}}
\newcommand{\bnfeq}{\mathrel{::=}}
\newcommand{\bnfor}{\ \big| \ }

\newcommand{\set}[1]{\{#1\}}
\newcommand{\setst}[2]{\{\ #1\ \boldsymbol{|}\ #2\ \}}

\newcommand{\powerset}[1]{2^{#1}}

\newcommand{\true}{\text{true}}
\newcommand{\false}{\text{false}}

\newcommand{\Scomb}{\mathtt{S}}
\newcommand{\Kcomb}{\mathtt{K}}

\newcommand{\agents}{\mathcal{A}}
\newcommand{\CM}{\mathtt{CM}}

\newcommand{\community}{\mathcal{C}}

\newcommand{\messages}{\mathcal{M}}
\newcommand{\data}{\mathcal{D}}

\newcommand{\pair}[2]{(#1,#2)}

\newcommand{\derives}[3]{#3\vdash_{#1}#2}

\newcommand{\states}{\mathcal{S}}

\newcommand{\gscc}[2]{\mathtt{succ}_{#1}^{#2}}
\newcommand{\msgs}[1]{\mathrm{msgs}_{#1}}

\newcommand{\indist}[3]{#2\equiv_{#1}#3}

\newcommand{\pAccess}[3]{\mathrel{_{#1}\negthinspace\mathrm{R}_{#2}^{#3}}}
\newcommand{\access}[3]{\mathrel{_{#1}\negthinspace\mathcal{R}_{#2}^{#3}}}

\newcommand{\clo}[2]{\mathrm{cl}_{#1}^{#2}}
\newcommand{\Clo}[2]{\mathrm{Cl}_{#1}^{#2}}

\newcommand{\pFormulas}{\mathcal{L}}

\renewcommand{\true}{\top}
\renewcommand{\false}{\bot}

\newcommand{\relsym}[1]{\thinspace{#1}\thinspace}
\newcommand{\knows}[2]{#1\relsym{\mathsf{k}}#2}

\newcommand{\limp}{\rightarrow}
\newcommand{\lequiv}{\leftrightarrow}

\newcommand{\pproves}[4]{#1\relsym{{:}_{#3}^{#4}}#2}

\newcommand{\decides}[4]{#1\relsym{{\veebar}_{#3}^{#4}}#2}

\newcommand{\aModalFrame}{\mathfrak{S}}

\usepackage{colortbl} 

\newcommand{\proves}[4]{#1\relsym{::_{#3}^{#4}}#2}

\newcommand{\LIiP}{\mathrm{LIiP}}
\newcommand{\LIiPded}{\vdash_{\LIiP}}
\newcommand{\LIiPdedBis}{\mathrel{{\dashv}{\vdash}_{\LIiP}}}

\newcommand{\Iproves}[4]{#1\relsym{{\pm}_{#3}^{#4}}#2}
\newcommand{\Iproofdiamond}[4]{#1\relsym{{\mp}_{#3}^{#4}}#2}

\newcommand{\LDiiP}{\mathrm{LDiiP}}

\newcommand{\LDiiPded}{\vdash_{\LDiiP}}

\newcommand{\canrel}[3]{\mathrel{_{#1}\negthinspace\mathrm{C}_{#2}^{#3}}}
\newcommand{\canVal}{\mathcal{V}_{\mathsf{C}}}
\newcommand{\canModel}{\mathfrak{M}_{\mathsf{C}}}

\begin{document}
\title{Logic of Intuitionistic Interactive Proofs\\
		(Formal Theory of Perfect Knowledge Transfer)\thanks{%
	Work partially funded with 
		Grant AFR~894328 from the National Research Fund Luxembourg cofunded under 
			the Marie-Curie Actions of the European Commission (FP7-COFUND) \cite{ILIiP}.}}
\author{Simon Kramer\\[\jot]
		\texttt{simon.kramer@a3.epfl.ch}}
\maketitle

\begin{abstract}
We produce a decidable \emph{super-intuitionistic} normal modal logic of 
	internalised \emph{intuitionistic} (and thus disjunctive and monotonic) interactive proofs (LIiP) from
		an existing classical counterpart of 
			classical monotonic non-disjunctive interactive proofs (LiP).
Intuitionistic interactive proofs effect 
	a durable epistemic impact in the possibly adversarial communication medium CM (which is imagined as a distinguished agent)   
		\emph{and only in that,} 
		that consists in the permanent induction of the \emph{perfect} and thus \emph{disjunctive} knowledge of their  
			proof goal by means of CM's knowledge of the proof: 
				If CM knew my proof 
				then CM would persistently and also \emph{disjunctively} know that my proof goal is true.
So intuitionistic interactive proofs effect 
	a lasting transfer of disjunctive propositional knowledge (disjunctively knowable facts) in 
		the communication medium of multi-agent distributed systems via 
		the transmission of certain individual knowledge (knowable intuitionistic proofs).
Our (necessarily) CM-centred notion of proof is also 
	a disjunctive explicit refinement of KD45-belief, and yields also such 
	a refinement of standard S5-knowledge.
Monotonicity but not communality is a commonality of LiP, LIiP, and their internalised notions of proof.
As a side-effect, 
	we offer a short internalised proof of 
			the Disjunction Property of 
				Intuitionistic Logic (originally proved by G\"odel).

\end{abstract}

\section{Introduction}\label{section:Introduction}
\paragraph{Subject matter} The subject matter of this paper is 
	normal modal logic of internalised monotonic interactive proofs, \ie 
		a novel super-intuitionistic normal modal logic of 
			internalised intuitionistic (and thus disjunctive and monotonic) interactive proofs (LIiP) as well as 
		an existing classical normal modal logic of 
			internalised classical monotonic (and thus non-disjunctive) interactive proofs (LiP) \cite{LiP,KramerICLA2013}.
(We abbreviate interactivity-related adjectives with lower-case letters.) 
Recall from \cite{sep-logic-intuitionistic} that 
	a super-intuitionistic propositional logic is 
		any consistent collection of propositional formulas 
			that contains all the axioms of Intuitionistic Propositional Logic (IL) and 
			that is closed under \emph{modus ponens} and substitution of arbitrary formulas for proposition letters.
Note however that the \emph{language} of IL is a strict subset of the (propositionally modal) language of LIiP.

\paragraph{Goal} Our goal here is to produce LIiP axiomatically as well as semantically from LiP.
The process of constructing LIiP from LiP is presented for the sake of gaining insight into 
	the semantic connection between intuitionistic and classical interactive proofs, respectively.
LIiP, the result of the construction, however is independent of LiP.
Further note that like in \cite{LiP}, \cite{LiiP,KramerICLA2013}, and \cite{LDiiP,KramerIMLA2013}, 
	we still understand interactive proofs as \emph{sufficient evidence} for 
		intended resource-unbounded proof-checking agents (who are though unable to guess), and 
			leave probabilistic and polynomial-time resource bounded agents 
				for future work.
Finally note that 
	we choose our meta-logic to be classical 
		(singleton meta-universe or meta-world unicity, \cf Section~\ref{section:IL}).

\subsection{Motivation}
Our immediate motivation for LIiP is to complete  
	the picture of 
		our above-mentioned resource-unbounded propositional normal modal logics of 
			interactive proofs with 
					the missing variant of intuitionistic interactive proofs---see Table~\ref{table:picture}.
\begin{sidewaystable}
\centering
\caption{Proof-term properties}
\smallskip
\begin{tabular}{@{}|>{\bf}>{\centering}p{0.075\textwidth}|c|>{\centering}p{0.3\textwidth}|>{\centering}p{0.15\textwidth}|c|@{}}
	\cline{2-5}
	\multicolumn{1}{c|}{} & 
		\bf classicality & 
		\bf communality & 
		\bf disjunctivity (constructivity) & 
		\bf monotonicity \\
	\hline
	\rowcolor[gray]{0.75}
	LIiP & 
		\textbf{\emph{intuitionistic}} & 
		non-communal 
			(\emph{communication-medium}\\ or \emph{-adversary} centred\\  
				\cite[Chapter~16]{ComputerSecurity}) & 
		\emph{disjunctive,} negation-\emph{in}complete & 
		monotonic \\
	\hline
	LiP \textmd{\cite{LiP,KramerICLA2013}} & 
		classical & 
		common knowledge 
			(security-infrastructure or \emph{meta}-modelling:  
				clocks \cite[Chapter~16]{CryptographyEngineering}, 
				names \cite[Chapter~6]{SecurityEngineering}, 
				PKIs \cite[Chapter~18--20]{CryptographyEngineering}) & 
		non-disjunctive, negation-\emph{in}complete & 
		monotonic \\	
	\hline	
	LiiP
	\textmd{\cite{LiiP,KramerICLA2013}} & 
		classical & 
		common belief\\ 
			(security \emph{object}-models: 
				access control 
					\cite[Chapter~4]{SecurityEngineering}, 
					\cite[Chapter~5, 20]{ComputerSecurity};  
				data-base security \cite[Chapter~9]{ComputerSecurity};  
				communication protocols 
					\cite[Chapter~3]{SecurityEngineering}, 
					\cite[Chapter~16]{ComputerSecurity},
					\cite[Chapter~13]{CryptographyEngineering}) & 
		non-disjunctive, negation-\emph{in}complete & 
		non-monotonic \\
	\hline
	LDiiP
	\textmd{\cite{LDiiP,KramerIMLA2013}} & 
		classical & 
		non-communal 
			(\emph{arbitrary-single-agent centred:}  
				\emph{network}-adversaries   
					\cite[Chapter~21]{SecurityEngineering}, 
					\cite[Chapter~17]{ComputerSecurity}; 
				reference monitors \cite[Chapter~6]{ComputerSecurity}) & 
		\emph{disjunctive,} \textbf{\emph{negation-complete}} & 
		non-monotonic \\
	\hline
\end{tabular}
\label{table:picture}
\end{sidewaystable}
The overarching motivation for LIiP is to serve in 
	an intuitionistic foundation of interactive computation.
See \cite{LiP} for a programmatic motivation.
Table~\ref{table:picture} displays characteristic properties of 
	our interactive proofs as internalised in 
		their respective resource-unbounded propositional normal modal logic together with
			typical applications in information security.
The logics themselves, in contrast to their internalised proof terms, except LIiP are classical, \ie 
	monotonic and non-disjunctive (and thus negation-incomplete).
As a confirmation, 
	notice that disjunctivity is a necessary but not sufficient condition for intuitionism.

We recall and explain all this logical terminology in the next subsections, and thereby 
draw some inspiration from the quite different intuitionistic logic of 
	intuitionistic non-interactive proofs \cite{LIPArtemov}\label{page:LIP:1} and  
	from the informational views on modal and intuitionistic logic expressed in 
		\cite{ModalLogicInformation,InformationIL}.

\subsubsection{Intuitionistic Logic (IL)}\label{section:IL}
\paragraph{Definition}
From \cite{sep-logic-intuitionistic}, 
			recall that 
				Intuitionistic Propositional Logic (IL) can be succinctly described as 
					Classical Propositional Logic without the Aristotelian law of excluded middle (LEM): 
						$(A \vee \neg A)$, but with the law of contradiction 
						$(\neg A \limp (A \limp B))$, and that 
		intuitionistically, \emph{Reductio ad absurdum} only proves negative statements, 
			since $(\neg\neg A\limp A)$ does not hold in general.
(If it did, LEM would follow by \emph{modus ponens} from 
	the intuitionistically provable $\neg\neg(A\lor\neg A)$.)
Semantically, 
	IL (and LIiP) is perhaps best viewed as a modal logic \cite{KripkeSemanticsIL} 
		(\cf Definition~\ref{definition:KripkeModel} and Table~\ref{table:SatisfactionRelation}).
Therein, 
	\begin{itemize}
	\item the valuation function on atomic propositions is constrained to be  
			monotonic with respect to a partial order on system states (possible worlds);\footnote{%
				Incidentally, this monotonicity makes intuitionistic logic incompatible with 
					hybrid logic, whose 
						nominals are atomic propositions true at a single state \cite{HybridLogics}, 
							at least in the basic case where 
								the intuitionistic-logical universe of worlds coincides with 
								the hybrid-logical one.
				For more complex cases, see for example the work of Torben Bra\"uner and Valeria de Paiva.}
	\item the positive intuitionistic connectives (conjunction, disjunction) are interpreted as 
			their classical counterparts 
				(which conserve the monotonicity of atomic propositions) on 
					the current state; 
	\item the negative intuitionistic connectives (negation, implication) are interpreted as 
			their non-monotonic classical counterparts on
					\emph{the upset of} the current state \emph{with respect to 
						the partial order} 
							(and thus \emph{are made to} conserve the monotonicity of atomic propositions).\label{page:Monotonicity}
	\end{itemize}
Hence first, 
	intuitionistic negation and implication can be viewed as 
		classical negation and implication prefixed by a (unary) modality that is interpreted by 
			a (binary) partial-order accessibility relation (\eg a temporal reachability relation),
				respectively; and 
	second, intuitionistic facts, be they positive or negative, are necessarily 
		monotonic (durable, forward invariant, lasting, persistent, stable) in the state space \cite{InformationIL}.\label{page:invariants}
(Intuitionistic double negation can be interpreted temporally as 
	the forward invariant ``at some future time.'')
Plain classical logic is also but trivially monotonic, as  
	it can be viewed as a modal logic over a singleton state space.
From this modal viewpoint, 
	one immediately recognises why $(\neg\neg A\limp A)$ 
	(``true at some future time implies true now'') is valid classically.
				
\paragraph{Properties}
In the previous paragraph,
	we saw that 
		IL has the monotonicity property (``\emph{intuitionistic} implies \emph{monotonic}'').
IL has also the important disjunction property (``\emph{intuitionistic} implies \emph{disjunctive}'').
That is,
	any external intuitionistic notion of proof $\vdash_{I}$ has 
		the property that 
			$\vdash_{I} A\lor B$ implies $\vdash_{I} A$ or $\vdash_{I} B$.
Recall that 
	disjunctivity is a necessary but not sufficient condition for intuitionism, and that
	plain classical notions of proof do not have the disjunction property.
Now note that when internalised in an object-logical language, 
	\begin{itemize}
		\item an external intuitionistic notion of proof, say $\vdash_{I_{1}}$, becomes 
			a unary necessity modality, say $[M]$, 
				parametrised with a proof term $M\,;$ 
		\item the disjunction property of $\vdash_{I_{1}}$ becomes 
				a disjunctive property of 
					the internalising external notion of proof, say $\vdash_{I_{2}}$, that 
						$\vdash_{I_{2}}[M](\phi\lor\phi')\limp([M]\phi\lor[M]\phi')\,;$
		\item the monotonicity property becomes
				the similar property that 
					$\vdash_{I_{2}}[M]\phi\limp[(M,M')]\phi$, where 
						$(M,M')$ is the term pair constructed from $M$ and $M'$ with 
							$M'$ representing the additional data allowed by the monotonicity.
	\end{itemize}
Further note that 
	a normal modal logic that internalises an intuitionistic notion of proof is necessarily intuitionistic itself: 
		assume that 
			the external notion of proof, say $\vdash$, is classical, \ie $\vdash\phi\lor\neg\phi$, and
			deduce $\vdash[M](\phi\lor\neg\phi)$ by the normal-modal rule schema of necessitation that
				$\vdash \phi$ implies $\vdash[M]\phi\,.$
So the internalised notion of proof $[M]$ is classical too. 
Even a non-normal modal logic like \cite{LIPArtemov}\label{page:LIP:2} must be intuitionistic itself in order to
		be able to internalise an intuitionistic notion of proof, because 
			already \cite{LIPArtemov}'s weaker form of necessitation---that 
				$\vdash_{\text{ILP}} F$ implies that 
				\emph{there is a} proof term $t$ such that 
					$\vdash_{\text{ILP}} t:F$---forces external intuitionism.
Finally note that 
	IL is algorithmically decidable \cite{ILcomplexity}.

\subsubsection{Negation-complete logics} 
\paragraph{Definition} 
Recall that 
	a notion of proof, say $\vdash_{1}$, is \emph{negation-complete} by definition if and only if 
		$\vdash_{1}A$ or $\vdash_{1}\neg A$.
When internalised in an object language with an external notion of proof, say $\vdash_{2}$, 
	negation completeness becomes 
		$\vdash_{2}[M]\phi\lor[M]\neg\phi$. 
Though conveniently classical (LEM), 
	negation-complete logics have also 
		the property of having constructive and computational content.
			
\paragraph{Properties} 
From the detailed reminder in Section~1.1.1 of \cite{LDiiP,KramerIMLA2013}, 
	recall that 
		first, 
			the negation-completeness property implies 
				the discussed disjunction property (``\emph{negation-complete} implies \emph{disjunctive}'');
		second, any internalised negation-complete notion of proof $\vdash_{2}$ is non-monotonic, that is, 
				$\vdash_{2}[M]\phi\lor[M]\neg\phi$ implies 
				$\not\vdash_{2}[M]\phi\limp[(M,M')]\phi$
					(``\emph{negation-complete} implies \emph{non-monotonic}'');
		third, negation completeness and intuitionism are incompatible properties; and 
		fourth, negation completeness implies algorithmic decidability.

\subsubsection{Communality}
In LDiiP, LiiP, LiP, and LIiP, 
	our so-far \emph{ad hoc} modal notation $[M]\phi$ becomes 
				$\decides{M}{}{a}{}\phi\,,$ 
				$\proves{M}{}{a}{\community}\phi\,,$
				$\pproves{M}{}{a}{\community}\phi\,,$ and
				$\Iproves{M}{}{\CM}{}\phi\,,$ respectively, where
					$a$ and $\community$ is an additional parameter for   
						a peer-reviewing agent $a$ 
							(such as the as-an-agent-imagined communication medium $\CM$) and 
						a finite agent-community $\community$ of peers, respectively.
The intended meaning of these modalities is 
	``$M$ can classically and disjunctively but only non-monotonically prove to $a$ that $\phi$ [is true],''
	``$M$ can classically and non-monotonically prove to $a$ that $\phi$ and
		this fact is common belief in $\community\cup\set{a}$,''
	``$M$ can classically and monotonically prove to $a$ that $\phi$ and 
		this fact is common knowledge in $\community\cup\set{a}$,'' and
	``$M$ can intuitionistically (and thus disjunctively and monotonically) prove to $\CM$ that $\phi$,'' respectively.
(Recall from
	\cite{Epistemic_Logic,MultiAgents}, 
		that knowledge implies belief.)
In all these logics,
	the proof potential is such that 
		if my peer reviewer knew my proof 
		then she would know that its proof goal is true.
Notice that 
	what is accepted as a potential proof $M$ 
		may depend on a community $\community\cup\set{a}$ of peers if and only if
			the proof is non-disjunctive.
This is the (non-)communality of $M$ mentioned in Table~\ref{table:picture}.

\subsection{Contribution}\label{section:Contribution}
Our contribution in this paper is five-fold:
\begin{enumerate}
	\item We produce the intuitionistic Logic of Intuitionistic interactive Proofs (LIiP)  
		(\cf Theorem~\ref{theorem:Adequacy}) from its classical counter-part LiP. 
		LIiP internalises necessarily \emph{communication-medium-centred} 
			(or \emph{communication-adversary-centred})\footnote{
										In communication security, 
											the communication medium is usually 
												assumed adversarial (as an agent) and 
												called Eve (the eavesdropper).} intuitionistic proof theories, 
			enjoying the disjunction property.
		%
		%
		As notable \emph{\textbf{syntactic} novelties} in intuitionistic modal logic, 
			LIiP provides: 
			\begin{enumerate}
				\item a non-primitive 
					\begin{enumerate}
						\item possibility modality that 
				is \emph{doubly macro-definable} within the language of LIiP: 
					in terms of double negation and
					\begin{enumerate}
						\item communication-medium knowledge (\cf Page~\pageref{page:PossibilityModality}), 
						\item its corresponding primitive necessity modality 
					(though \emph{not} in the classical modal terms $\Diamond\phi\lequiv\neg\Box\neg\phi$, \cf Theorem~\ref{theorem:SomeUsefulDeducibleLogicalLaws}.58);
					\end{enumerate}
						\item necessitation rule that 
								is derivable from 
									the primitive modal monotonicity axiom schema $\phi\limp\Box\phi$ and 
									the primitive \emph{modus ponens} deduction rule 
										(\cf Theorem~\ref{theorem:SomeUsefulDeducibleLogicalLaws}.0);
					\end{enumerate}
				\item the two insights that in interactive settings, 
					\begin{enumerate}
						\item the intuitionistic truths are those of the communication medium 
								(\cf Remark~\ref{remark:IntuitionisticTruths}),
						\item intuitionistic proofs induce perfect and thus disjunctive knowledge
								in the communication medium, and only in that 
									(\cf Remark~\ref{remark:PerfectKnowledge}).
					\end{enumerate}
			\end{enumerate}
			That is, we put forward 
				\emph{LIiP as a modal-logical characterisation of 
						the concept of message-passing communication medium.}
	\item We provide a standard but also oracle-computational and set-theoretically constructive Kripke-semantics for LIiP  
		(\cf Section~\ref{section:Semantically}):
		\begin{itemize}
	\item Like in \cite{LiiP,KramerICLA2013} and \cite{LDiiP,KramerIMLA2013},
		we endow the proof modality with a standard Kripke-semantics \cite{ModalLogicSemanticPerspective}, but first define its 
			accessibility relation $\access{M}{\CM}{}$  
				constructively in terms of 
					elementary set-theoretic constructions,\footnote{in loose analogy with 
							the set-theoretically constructive rather than 
							the purely axiomatic definition of 
								numbers \cite{TheNumberSystems} or 
								ordered pairs (\eg the now standard definition by Kuratowski, 
									and other well-known definitions \cite{NotesOnSetTheory})} 
							namely as $\pAccess{M}{\CM}{}$ (\cf Section~\ref{section:Concretely}),
				and then match it to an abstract semantic interface in standard form (which 
					abstractly stipulates the characteristic properties of the accessibility relation
						\cite{ModalProofTheory}).
	We will say that $\pAccess{M}{\CM}{}$ \emph{exemplifies} 
		(or \emph{realises}) $\access{M}{\CM}{}$ (\cf Section~\ref{section:Abstractly}).
	(A simple example of a set-theoretically constructive but non-intuitionistic definition of a modal accessibility is
		the well-known definition of epistemic accessibility as 
			state indistinguishability defined in terms of  
				equality of state-projection functions \cite{Epistemic_Logic}.)
	
	\item 	Our Kripke-semantics is oracle-computational in the sense that 
		the individual proof knowledge (say $M$) can be thought of as being provided by 
			a computation oracle (\cf Definition~\ref{definition:SemanticIngredients}), which thus acts as a 
				hypothetical provider and imaginary epistemic source of our interactive proofs.
	\end{itemize}
	As notable \emph{\textbf{semantic} novelties} in intuitionistic modal logic, 
		LIiP is: 
		\begin{enumerate}
			\item \emph{doubly constructive:} 
			LIiP is an intuitionistic (and thus constructive) logic and additionally  
			offers a set-theoretically constructive Kripke-semantics in the form of  
				the concrete accessibility relation $\pAccess{M}{\CM}{}\,;$
			\item \emph{parametrically mono-relational:}
			We may freely choose between 
				the abstract accessibility $\access{M}{\CM}{}$ and 
				its concrete exemplification (or realisation) $\pAccess{M}{\CM}{}$ as 
					the accessibility relation in  
						LIiP's Kripke-semantics, but 
							$\pAccess{\CM}{\CM}{}$ has also the property of 
								being a partial order with  
									$\CM$ designating the communication medium.
		Hence LIiP's essentially mono-relational models subsume 
			seminal bi-relational models of intuitionistic modal logics \cite[Page~59]{PhDThesisSimpson}  
				by absorbing the partial order $\pAccess{\CM}{\CM}{}$ for the Kripke-semantics of LIiP's intuitionistic connectives 
					as a mere instance of the accessibility relation $\pAccess{M}{\CM}{}$ for the Kripke-semantics of LIiP's proof modality\label{page:PartialOrder},
						thanks to being parametric (and thus generic).
		\end{enumerate}
	\item We prove a \emph{modal-depth result} applying both to 
			LIiP's necessity as well as 
			its corresponding possibility modality (\cf Corollary~\ref{corollary:ModalDepth}).
	\item We prove that 
		our CM-centred notion of proof is also 
			a \emph{disjunctive explicit} refinement of standard KD45-belief, and yields also such 
			a refinement of standard S5-knowledge and S4-provability 
				(\cf Corollary~\ref{corollary:DEB} and \ref{corollary:DEP}).
	\item We prove 
		the \emph{finite-model property} (\cf Theorem~\ref{theorem:FiniteModelProperty}) and therefrom 
		the \emph{algorithmic decidability} of LIiP (\cf  Corollary~\ref{corollary:AlgorithmicDecidability}).
\end{enumerate}
As a side-effect of our work on LIiP, 
			we offer an internalised, 
				three-line two-axiom proof of 
					the Disjunction Property of 
						Intuitionistic Logic (IL) 
							originally proved by G\"odel.
			The two axioms 
				are modal internalisations of two fundamental properties of IL, namely 
					the truthfulness of its proofs and
					Kripke's Monotonicity Lemma for his semantics of IL.
			Surprisingly, 
				they jointly 
					trivialise the corresponding box modality `$\Box$' (though \emph{not} the one of LIiP) in a technical sense, and thus also,
						in a non-technical sense, 
							G\"odel's (non-trivial) proof of IL's Disjunction Property.
			The truthfulness of intuitionistic proofs corresponds to 
				the well-known modal T-law $\Box\phi\limp\phi$ and 
			Kripke's Monotonicity Lemma to
				the law $\phi\limp\Box\phi$
				(not to be confused with the mentioned monotonicity of proof terms,
					\cf Page~\pageref{page:Monotonicity} and Remark~\ref{remark:Monotonicity}).
			Jointly, 
				they imply $\Box(\phi\lor\varphi)\limp(\Box\phi\lor\Box\varphi)$ in any modal logic:
				\begin{enumerate}
					\item $\vdash\Box(\phi\lor\varphi)\limp(\phi\lor\varphi)$\hfill T
					\item $\vdash(\phi\lor\varphi)\limp(\Box\phi\lor\Box\varphi)$\hfill 
							$\vdash\phi\limp\Box\phi$, 
							$\vdash\varphi\limp\Box\varphi$, IL
					\item $\vdash\Box(\phi\lor\varphi)\limp(\Box\phi\lor\Box\varphi)$\hfill 1, 2, IL.
				\end{enumerate}

\subsection{Roadmap}
In the next section, 
	we introduce our Logic of Intuitionistic interactive Proofs (LIiP) axiomatically by means of 
		a compact closure operator that induces the Hilbert-style proof system that we seek. 
We then prove 
	a substantial number of useful, deducible structural and logical laws 
		(\cf Theorem~\ref{theorem:SomeUsefulDeducibleStructuralLaws} and 
						\ref{theorem:SomeUsefulDeducibleLogicalLaws}) within the obtained system, and 
	therefrom important corollaries 
		(Corollary~\ref{corollary:Combinator}--\ref{corollary:DEP}), some of which
			count as our aforementioned contributions in this paper.
%
%
Next, we 
	introduce the concretely constructed semantics 
	as well as the standard abstract semantic interface for LIiP (\cf Section~\ref{section:Semantically}), and
	prove the axiomatic adequacy of the proof system with respect to this interface 
		(\cf Theorem~\ref{theorem:Adequacy}).
In the construction of the semantics, we again make use of a closure operator, but 
	this time on sets of proof terms. 
Finally,
	we prove 
		the finite-model property 
			(\cf Theorem~\ref{theorem:FiniteModelProperty}) and  
		the algorithmic decidability 
			(\cf Corollary~\ref{corollary:AlgorithmicDecidability}) of LDiiP.

\section{LIiP}
\subsection{Syntactically}
The Logic of Intuitionistic interactive Proofs (LIiP) provides
	a modal \emph{formula language} over a generic message \emph{term language}.
The formula language offers
		the propositional constructors, 
		a relational symbol 
			`$\knows{}{}$' 
			for constructing atomic propositions about 
				so-called individual knowledge (\eg $\knows{a}{M}$), and
		a modal constructor `$\Iproves{}{}{}{}$' for propositions about proofs
			(\eg $\Iproves{M}{\phi}{\CM}{}$).
The message language 
	offers a term constructor for message \emph{pairing} and 
	can accommodate arbitrary other term constructors, \eg for cryptography (\cf \cite{LiP}).
The single term constructor of pairing is sufficient for 
	internalising \emph{modus ponens} into the message language 
		(\cf Theorem~\ref{theorem:SomeUsefulDeducibleLogicalLaws}.1) and 
			thus for internalising the single deduction rule of intuitionistic logic into it.
\emph{Modus ponens} can be regarded as a minimal requirement for a system to count as a proof system.
And so in the context of LIiP, 
	all other term constructors can be regarded as application-specific, and 
		LIiP-theories with such constructors as applied LIiP-theories.
These however are not the subject matter of our present paper about basic (or pure) LIiP.
(Message signing has no essential role in LIiP as opposed to LiP.)
In brief, LIiP is a minimal modular extension of IL with an 
interactively generalised
	additional operator 
	(the proof modality) and 
	proof-term language (only one, binary built-in constructor; \emph{agents as proof-checkers}).
Alternatively,
	LIiP can be viewed as a refinement (due to its parameterised modality) and extension (due to additional laws) of 
			Fischer Servi's \cite{IK:FischerServi} or, equivalently,  
			Plotkin and Stirling's \cite{IK:PlotkinStirling} 
		basic intuitionistic modal logic IK, 
			promoted in \cite{PhDThesisSimpson} as ``the true intuitionistic analogue of K.''
See \cite{BasicConstructiveModality} for a recent 
	discussion of this purported truth and 
	alternative contribution in the form of the so-called basic constructive modal logic CK, which
		can be embedded into IK \cite{CKembeddedIntoIK}.
Note that the formula language of LIiP is identical to the one of LiP \cite{LiP}
	modulo 
	the term language and 
	the proof-modality notation.
The term language of LIiP is strictly included in the term language of LiP but may be arbitrarily extended.
The proof-modality notation in LIiP is `$\Iproves{}{}{}{}$' whereas it is `$\pproves{}{}{}{}$' in LiP.

In the sequel, grey-shading indicates essential differences to LiP. 
\begin{definition}[The language of LIiP]\label{definition:LIiPLanguage}
	Let
	\begin{itemize}
		\item $\agents$ designate a finite set of 
			\emph{agent names} $a$, $b$, $c$, \etc \colorbox[gray]{0.75}{such that $\CM\in\agents$,}
				where $\CM$ designates the communication medium 
					(admissible also in LiP);
		\item $\messages \ni M \bnfeq a \bnfor B \bnfor \pair{M}{M}$
			designate our language of 
				\emph{message terms} $M$ \emph{over} $\agents$ with 
				(transmittable) agent names $a\in\agents$, 
				application-specific data $B$ (left blank here), and 
				message-term pairs $\pair{M}{M}\,;$
				
			(Messages must be grammatically well-formed, which
				yields an induction principle.
				So agent names $a$ are logical term constants, 
				the meta-variable $B$ just signals the possibility of an extended term language $\messages$, and 
				$\pair{\cdot}{\cdot}$ is a binary functional symbol. 
				For other term constructors, see \cite{LiP}.)
		\item $\mathcal{P}$ designate a denumerable set of \emph{propositional variables} $P$ constrained such that
			for all $a\in\agents$ and $M\in\messages$, 
				$(\knows{a}{M})\in\mathcal{P}$ 
					(for ``$a$ knows $M$'')  
					is a distinguished variable, \ie 
						an \emph{atomic proposition}, 
							(for \emph{individual} knowledge);
			(So, for $a\in\agents$, $\knows{a}{\cdot}$ 
				is a unary relational symbol.)
		\item $\pFormulas\ni\phi \bnfeq P \bnfor 
				\phi\land\phi \bnfor 
				\phi\lor\phi \bnfor 
				\neg\phi \bnfor 
				\phi\limp\phi \bnfor 
				\colorbox[gray]{0.75}{$\Iproves{M}{\phi}{\CM}{}$}$ 
				designate our language of \emph{logical formulas} $\phi$, where
					the modal-necessity formula $\Iproves{M}{\phi}{\CM}{}$ reads as  
					``$M$ can intuitionistically prove that $\phi$ (is true) to $\CM$.'' 
	\end{itemize}
\end{definition}
\noindent	
Note the following macro-definitions: 
	$\true \defeq \knows{\CM}{\CM}$, 
	$\false \defeq \neg \true$,  
	$\phi \lequiv \phi' \defeq (\phi \limp \phi') \land (\phi' \limp \phi)$, 
	\colorbox[gray]{0.75}{$\blacksquare\phi\defeq\Iproves{\CM}{\phi}{\CM}{}$}
		(see also Theorem~\ref{theorem:SomeUsefulDeducibleLogicalLaws}.8), 
	$\Diamond\phi\defeq\neg\neg\phi$, and 
	\colorbox[gray]{0.75}{$\Iproofdiamond{M}{\phi}{\CM}{}\defeq\Diamond(\knows{\CM}{M}\land\phi)$}
		(\emph{double negation as modal possibility} rather than necessity, unlike in  
		\cite{IntuitionisticDoubleNegationAsNecessity}\label{page:PossibilityModality}; 
			see also Theorem~\ref{theorem:SomeUsefulDeducibleLogicalLaws}.58).
Recall that
	whereas 
		conjunction and disjunction connectives as well as 
		necessitation and possibility modalities 
			are dually inter-definable in classical modal logic by means of negation, 
				they are not necessarily so in intuitionistic modal logic 
					\cite[Requirement~5]{PhDThesisSimpson}.
However,
	as our above macro-definition of $\Iproofdiamond{M}{\phi}{\CM}{}$ in terms of a double negation and individual knowledge foreshadows, 
		$\Iproofdiamond{M}{\phi}{\CM}{}$ fortunately is not necessary as a primitive modality in the language of LIiP.
(\cite{LIPArtemov}\label{page:LIP:4} remain silent as to the dual of their intuitionistic-proof modality.)

Then, LIiP has the following axiom and deduction-rule schemas. 
\begin{definition}[The axioms and deduction rules of LIiP]\label{definition:AxiomsRules}
Let
	\begin{itemize}
		\item $\Gamma_{0}$ designate an adequate set of axioms for \colorbox[gray]{0.75}{intuitionistic} propositional logic
		\item $\Gamma_{1} \defeq \Gamma_{0} \cup \{$
			\begin{itemize}
				\item $\knows{a}{a}$\quad(knowledge of one's own name string)
				\item $(\knows{a}{M}\land\knows{a}{M'})\lequiv\knows{a}{\pair{M}{M'}}$\quad([un]pairing)
				\item $\Iproves{M}{\knows{\CM}{M}}{\CM}{}$\quad(self-knowledge)
				\item $(\Iproves{M}{(\phi\limp\phi')}{\CM}{})\limp
				((\Iproves{M}{\phi}{\CM}{})\limp\Iproves{M}{\phi'}{\CM}{})$%
				\quad(K)
				\item $(\Iproves{M}{\phi}{\CM}{})\limp(\knows{\CM}{M}\limp\phi)$\quad(epistemic T, ET)	
				\item \colorbox[gray]{0.75}{$
							(\Iproves{M}{\phi}{\CM}{})\limp\Iproofdiamond{M}{\phi}{\CM}{}
						\text{\quad(intuitionistic D, ID)}$}%
				\item \colorbox[gray]{0.75}{$\phi\limp\Iproves{M}{\phi}{\CM}{}$
						\quad(modal monotonicity, MM)} 
				\}
			\end{itemize}
			designate a set of \emph{axiom schemas.} 
	\end{itemize}
	Then, $\colorbox[gray]{0.75}{$\LIiP\defeq\Clo{}{}(\emptyset)$}\defeq\bigcup_{n\in\mathbb{N}}\Clo{}{n}(\emptyset)$, where 
		for all $\Gamma\subseteq\pFormulas:$
		\begin{eqnarray*}
					\Clo{}{0}(\Gamma) &\defeq& \Gamma_{1}\cup\Gamma\\
					\Clo{}{n+1}(\Gamma) &\defeq& 
						\begin{array}[t]{@{}l@{}}
							\Clo{}{n}(\Gamma)\ \cup\\
							\setst{\phi'}{\set{\phi,\phi\limp\phi'}\subseteq\Clo{}{n}(\Gamma)}\cup
								\quad\text{(\emph{modus ponens}, MP)}\\
							\setst{(\Iproves{M'}{\phi}{\CM}{})\limp\Iproves{M}{\phi}{\CM}{}}{(\knows{\CM}{M}\limp\knows{\CM}{M'})\in\Clo{}{n}(\Gamma)}\\
							\quad\text{(epistemic antitonicity, EA)}.\label{page:EpistemicAntitonicity}
						\end{array}
				\end{eqnarray*}
		We call $\LIiP$ a \emph{base theory,} and
		$\Clo{}{}(\Gamma)$ an \emph{LIiP-theory} for any $\Gamma\subseteq\pFormulas$.
\end{definition}
\noindent
Notice the 
	logical order of LIiP, which like LiP's is, 
		due to propositions about (proofs of) propositions, \emph{higher-order propositional}.
\paragraph{Inherited laws} 
From LiP \cite{LiP},  
	we recall the discussion of 
		the (un)pairing axiom,  
		Kripke's law (K), 
		the laws of 
			epistemic T (ET) and 
			epistemic antitonicity (EA):
We assume the existence of a pairing mechanism modelling finite sets.
Such a mechanism is required by the important application of
	communication (not only cryptographic) protocols \cite[Chapter~3]{SecurityEngineering}, in which
		concatenation of high-level data packets is associative, commutative, and idempotent.
The key to the validity of K is that 
	we understand interactive proofs as sufficient evidence for
		intended resource-unbounded proof-checking agents (who are though still unable to guess).
Clearly for such agents,
	if 
		$M$ is sufficient evidence for $\phi\limp\phi'$ and $\phi$
	then so is $M$ for $\phi'$.
Then, the significance of ET 
	(which as opposed to the standard T-law is conditioned on individual knowledge) to interactivity is that 
	in truly distributed multi-agent systems, 
		not all proofs are known by all agents, \ie 
			agents are not omniscient with respect to messages. 
Otherwise, why communicate with each other?
So there being a proof does not imply knowledge of that proof. 
When an agent $a$ does not know the proof and 
the agent cannot generate the proof \emph{ex nihilo} herself by guessing it, 
only communication from a peer, who thus acts as an oracle, can entail the knowledge of the proof with $a$. 
%
%
Finally,
	note that 
		the law of self-knowledge is a theorem but not an axiom in LiP, and
	observe that EA is a rule of \emph{logical modularity} that 
			allows the modular generation of structural modal laws from 
				implication term laws  
					(\cf Theorem~\ref{theorem:SomeUsefulDeducibleStructuralLaws}).

\paragraph{New laws} We continue to discuss the new laws of LIiP, which are all axiom schemas: 
In contrast to LiP \cite{LiP}, 
	LIiP must have an \emph{intuitionistic} rather than a classical propositional axiom base $\Gamma_{0}$ as
		already explained at the end of Section~\ref{section:IL}.
Next, ID says that intuitionistic necessity (``box'') implies intuitionistic possibility (``diamond'').
As opposed to its classical-modal-logic equivalent D, 
	ID cannot be alternatively stated in the shape of 
		$\vdash\Box\phi\limp\neg\Box\neg\phi$.
Then, the law of modal monotonicity MM reflects 
	the semantic fact mentioned in Section~\ref{section:Contribution} that
		LIiP's Kripke-model absorbs 
			the partial order for 
				the Kripke-semantics of LIiP's intuitionistic connectives as 
					a mere instance of the accessibility relation for 
				the Kripke-semantics of LIiP's proof modality,
		thanks to being parametric (and thus generic, \cf Section~\ref{section:Semantically}).
Thus we adopt formulas of the shape $\phi\limp\Box\phi$ as axioms MM in 
	our intuitionistic modal logical system LIiP, 
		like Do\v{s}en in his intuitionistic modal logical system $\mathit{Hdn\Box}$ 
			\cite{IntuitionisticDoubleNegationAsNecessity}, where 
				he adopts formulas of that form as axioms \emph{dn2}.
To our knowledge,
	$\mathit{Hdn\Box}$ and LIiP are the only intuitionistic modal logics with such axioms.
Finally,
	we could add 
		admissible but not derivable rules and 
		their corresponding internalising axioms to 
			LIiP in the style of \cite{LIPArtemov}.\label{page:LIP:5} 
(Recall from \cite{sep-logic-intuitionistic} that 
	the admissible rules of a theory are the rules under which the theory is closed.
	Hence the set of primitive and derivable rules is a subset of the set of admissible rules, 
		and [but not] vice versa in classical [intuitionistic] logic.
	[IL is \emph{structurally incomplete}.]
	See \cite{jerabek:independent-bases} for suitable \emph{bases} of admissible rules.)
However,
	since the addition of admissible but not derivable rules to a theory 
		does not change the theory, such 
			an addition can be considered as unnecessary, 
				at least in our base theory.

In the sequel, ``:iff'' abbreviates ``by definition, if and only if''.
\begin{proposition}[Hilbert-style proof system]\label{proposition:Hilbert}
	Let 
		\begin{itemize}
			\item $\Phi\LIiPded\phi$ \text{:iff} if $\Phi\subseteq\LIiP$ then $\phi\in\LIiP$ 
			\item $\phi\LIiPdedBis\phi'$ \text{:iff} $\set{\phi}\LIiPded\phi'$ and $\set{\phi'}\LIiPded\phi$
			\item $\LIiPded\phi$ \text{:iff} $\emptyset\LIiPded\phi.$
		\end{itemize}
		In other words, ${\LIiPded}\subseteq\powerset{\pFormulas}\times\pFormulas$ is a \emph{system of closure conditions} in the sense of 
				\cite[Definition~3.7.4]{PracticalFoundationsOfMathematics}.
		For example:
			\begin{enumerate}
				\item for all axioms $\phi\in\Gamma_{1}$, $\LIiPded\phi$
				\item for \emph{modus ponens}, $\set{\phi,\phi\limp\phi'}\LIiPded\phi'$
				\item for epistemic antitonicity,
				
					$\set{\knows{\CM}{M}\limp\knows{\CM}{M'}}\LIiPded
						(\Iproves{M'}{\phi}{\CM}{})\limp\Iproves{M}{\phi}{\CM}{}$.
			\end{enumerate}
		(In the space-saving, horizontal Hilbert-notation ``$\Phi\LIiPded\phi$'', 
			$\Phi$ is not a set of hypotheses but a set of premises, \cf  
				\emph{modus ponens} and epistemic antitonicity.)
	
	Then, $\LIiPded$ can be viewed as being defined by 
		a $\Clo{}{}$-induced Hilbert-style proof system.
	In fact 
			${\Clo{}{}}:\powerset{\pFormulas}\rightarrow\powerset{\pFormulas}$ is a \emph{standard consequence operator,} \ie
				a \emph{substitution-invariant compact closure operator.}
\end{proposition}
\begin{proof}
	Like in \cite{LiP}.
	That a Hilbert-style proof system can be viewed as induced by 
		 a compact closure operator is well-known (\eg see \cite{WhatIsALogicalSystem});
	that $\Clo{}{}$ is indeed such an operator can be verified by 
		inspection of the inductive definition of $\Clo{}{}$; and
	substitution invariance follows from our definitional use of axiom \emph{schemas}.\footnote{%
		Alternatively to axiom schemas,
		we could have used 
			axioms together with an
			additional substitution-rule set
				$\setst{\sigma[\phi]}{\phi\in\Clo{}{n}(\Gamma)}$
		in the definiens of $\Clo{}{n+1}(\Gamma)$.}
\end{proof}

\begin{corollary}[Normality]\label{corollary:Normality}
		LIiP is a normal modal logic.
\end{corollary}
\begin{proof}
	Jointly by  
		Kripke's law and  
		\emph{modus ponens} (by definition), 
		necessitation (\cf proof of Theorem~\ref{theorem:SomeUsefulDeducibleLogicalLaws}.0), and 
		substitution invariance (\cf Proposition~\ref{proposition:Hilbert}).
\end{proof}
\noindent

We are now going to present some useful  
	deducible \emph{structural} laws of LIiP, 
		including the deducible non-structural rule of epistemic bitonicity, 
			used in the deduction of some of them.
Here, 
	``structural'' means 
	``deducible exclusively from term axioms.''
The laws are enumerated in a (total) order that respects (but cannot reflect) their respective proof prerequisites.
The laws are also deducible in LiP, in the same order and 
	without non-intuitionistic machinery \cite{LiP}.
\begin{theorem}[Some useful deducible structural laws]\label{theorem:SomeUsefulDeducibleStructuralLaws}\ 
	\begin{enumerate}
		\item $\LIiPded\knows{a}{\pair{M}{M'}}\limp\knows{a}{M}$
		\\(left projection, 
			1-way $\Kcomb$-combinator property)
		\item $\LIiPded\knows{a}{\pair{M}{M'}}\limp\knows{a}{M'}$\quad(right projection)
		\item $\LIiPded\knows{a}{\pair{M}{M}}\lequiv\knows{a}{M}$\quad(pairing idempotency)
		\item $\LIiPded\knows{a}{\pair{M}{M'}}\lequiv\knows{a}{\pair{M'}{M}}$\quad(pairing commutativity)
		\item $\LIiPded(\knows{a}{M}\limp\knows{a}{M'})\lequiv(\knows{a}{\pair{M}{M'}}\lequiv\knows{a}{M})$
		\\(neutral pair elements)
		\item $\LIiPded\knows{a}{\pair{M}{a}}\lequiv\knows{a}{M}$\quad(self-neutral pair element)
		\item $\LIiPded\knows{a}{\pair{M}{\pair{M'}{M''}}}\lequiv\knows{a}{\pair{\pair{M}{M'}}{M''}}$\quad(pairing associativity)
		\item $\set{\knows{\CM}{M}\lequiv\knows{\CM}{M'}}\LIiPded
					(\Iproves{M}{\phi}{\CM}{})\lequiv\Iproves{M'}{\phi}{\CM}{}$\quad(epistemic bitonicity)
		\item $\LIiPded(\Iproves{M}{\phi}{\CM}{})\limp\Iproves{\pair{M'}{M}}{\phi}{\CM}{}$\quad(proof extension, left)
		\item $\LIiPded(\Iproves{M}{\phi}{\CM}{})\limp\Iproves{\pair{M}{M'}}{\phi}{\CM}{}$\quad(proof extension, right)
		\item $\LIiPded((\Iproves{M}{\phi}{\CM}{})\lor\Iproves{M'}{\phi}{\CM}{})\limp\Iproves{\pair{M}{M'}}{\phi}{\CM}{}$\quad(proof extension)
		\item $\LIiPded(\Iproves{\pair{M}{M}}{\phi}{\CM}{})\lequiv\Iproves{M}{\phi}{\CM}{}$\quad(proof idempotency)
		\item $\LIiPded(\Iproves{\pair{M}{M'}}{\phi}{\CM}{})\lequiv\Iproves{\pair{M'}{M}}{\phi}{\CM}{}$\quad(proof commutativity)
	\item $\set{\knows{\CM}{M}\limp\knows{\CM}{M'}}\LIiPded(\Iproves{\pair{M}{M'}}{\phi}{\CM}{})\lequiv\Iproves{M}{\phi}{\CM}{}$
	\\(neutral proof elements)
		\item $\LIiPded(\Iproves{\pair{M}{\CM}}{\phi}{\CM}{})\lequiv\Iproves{M}{\phi}{\CM}{}$\quad(self-neutral proof element)
		\item $\LIiPded(\Iproves{\pair{M}{\pair{M'}{M''}}}{\phi}{\CM}{})\lequiv\Iproves{\pair{\pair{M}{M'}}{M''}}{\phi}{\CM}{}$
		\\(proof associativity)
	\end{enumerate}
\end{theorem}
\begin{proof} 
	Like in \cite{LiP}---no non-intuitionistic machinery is required.
\end{proof}
\noindent
For a discussion of these LIiP laws, 
		consider our discussion of their analogs in LiP \cite{LiP} and the following remark.
\begin{remark}[Monotonicity---Proof \& Truth]\label{remark:Monotonicity}
The law of proof extension captures 
	the monotonicity of the \emph{proof terms} in LIiP mentioned in Table~\ref{table:picture}, 
		as does its analog in LiP.
In contrast,
	the law of modal monotonicity (\cf Definition~\ref{definition:AxiomsRules}), 
		which does not hold in LiP,	
			captures the monotonicity of the \emph{local truths} in LIiP.
Recall from Section~\ref{section:Introduction} that in an intuitionistic (modal) universe (such as LIiP's), 
	all, \ie positive or negative (whence the notation `$\Iproves{}{}{\CM}{}$\negthinspace'), (local) truths are monotonic.
Whereas in a classical (modal) universe (such as LiP's), 
	not all (local) truths need be monotonic.
If a proof term is monotonic then its proof goal is.
If a proof goal is monotonic then its proof must be.
In LIiP, 
	all proof goals are monotonic, 
		thanks to LIiP being intuitionistic, which is what 
			forces them to be so.
However in LiP, 
	not all proof goals need be monotonic, 
		because of LiP being classical as well as modal, which is what 
			frees them from being so.
\end{remark}
\begin{corollary}[$\Scomb$-combinator property]\label{corollary:Combinator}\  
	\begin{enumerate}
		\item $\LIiPded\knows{a}{\pair{\pair{M}{M'}}{M''}}\lequiv\knows{a}{\pair{M}{\pair{M''}{\pair{M'}{M''}}}}$
		\item $\LIiPded(\Iproves{\pair{\pair{M}{M'}}{M''}}{\phi}{\CM}{})\lequiv
						\Iproves{\pair{M}{\pair{M''}{\pair{M'}{M''}}}}{\phi}{\CM}{}$
	\end{enumerate}
\end{corollary}
\begin{proof}
	Like in \cite{LiP}---again, no non-intuitionistic machinery is required.	
\end{proof}

We are going to present also some useful, deducible \emph{logical} laws of LIiP.
Here, 
	``logical'' means 
	``not structural'' in the previously defined sense.
Also these laws are enumerated in an order that respects their respective proof prerequisites.
%
%
Grey-shading indicates special interest for  
	intuitionistic modal logic in general and for LIiP as opposed to (the classical) LiP in particular.
Three important themes therein are: 
	first, intuitionistic negation (single `$\neg$' and double `$\neg\neg$') and 
	second, individual knowledge ($\knows{\CM}{M}$), and their import for 
		the relation between 
			$\Iproves{M}{\phi}{\CM}{}$ and its dual 
			$\Iproofdiamond{M}{\phi}{\CM}{}$, which 
				is normally not one of identity nor dual definability in intuitionistic modal logic; and 
	third, the internalised disjunction property (IDP), which
		does not hold in LiP.
(\cite{LIPArtemov}\label{page:LIP:7} remain silent about the deducibility of an IDP in their logic, which, given that they internalise standard IL, is intriguing.)
Theorem~\ref{theorem:SomeUsefulDeducibleLogicalLaws} has four important corollaries, among which there is 
	a modal-depth result, 
	the relation of LIiP to 
		Fischer Servi's \cite{IK:FischerServi} and  
		Plotkin and Stirling's \cite{IK:PlotkinStirling} seminal work on intuitionistic modal logic, and 
	the relation of LIiP to standard 
		doxastic \cite{MultiAgents} and 
		epistemic logic \cite{MultiAgents,Epistemic_Logic,EpistemicLogicFiveQuestions}.
The number of intermediate results required to obtain the corollaries, 
	reflected in the length of Theorem~\ref{theorem:SomeUsefulDeducibleLogicalLaws}, 
		may be indicative of the exponential blow-up in proof length of 
			intuitionistic over classical logic \cite{LowerProofComplexityIL}.
The non-intuitionistically inclined reader may want to skip them except 
	Theorem~\ref{theorem:SomeUsefulDeducibleLogicalLaws}.58.
Whereas the intuitionistically inclined reader may want to prove them herself, 
	in order and as milestones for proving their corollaries.
\begin{theorem}[Some useful deducible logical laws]\label{theorem:SomeUsefulDeducibleLogicalLaws}\  
		\begin{enumerate}\setcounter{enumi}{-1}
		\item $\{\phi\}\LIiPded\Iproves{M}{\phi}{\CM}{}$\quad(necessitation, N)
		\item $\LIiPded(\Iproves{M}{(\phi\limp\phi')}{\CM}{})\limp
				((\Iproves{M'}{\phi}{\CM}{})\limp\Iproves{\pair{M}{M'}}{\phi'}{\CM}{})$\\
				\quad(generalised Kripke-law, GK)
		\item $\set{\phi\limp\phi'}\LIiPded(\Iproves{M}{\phi}{\CM}{})\limp\Iproves{M}{\phi'}{\CM}{}$\quad(regularity, R)		
		\item $\set{\phi\lequiv\phi'}\LIiPded(\Iproves{M}{\phi}{\CM}{})\lequiv\Iproves{M}{\phi'}{\CM}{}$\quad(R \emph{bis})
		\item $\set{\knows{\CM}{M}\limp\knows{\CM}{M'},\phi\limp\phi'}\LIiPded(\Iproves{M'}{\phi}{\CM}{})\limp\Iproves{M}{\phi'}{\CM}{}$\\
		\quad(epistemic regularity, ER)		
		\item $\set{\knows{\CM}{M}\lequiv\knows{\CM}{M'},\phi\lequiv\phi'}\LIiPded(\Iproves{M'}{\phi}{\CM}{})\lequiv\Iproves{M}{\phi'}{\CM}{}$\quad(ER \emph{bis})
		\item $\LIiPded\Iproves{M}{\true}{\CM}{}$\quad(anything can prove tautological truth)		
		\item $\LIiPded\Iproofdiamond{M}{\true}{\CM}{}$\quad(anything can disprove tautological falsehood)	
		\item \colorbox[gray]{0.75}{$\LIiPded\blacksquare\phi\lequiv\phi$\quad(TMM)}
		\item $\phi\LIiPdedBis\blacksquare\phi$\quad(TMM \emph{bis})
		\item \colorbox[gray]{0.75}{$\LIiPded\knows{\CM}{M}\limp((\Iproves{M}{\phi}{\CM}{})\lequiv\phi)$\quad(ET \emph{bis})}
		\item \colorbox[gray]{0.75}{$\LIiPded\neg\neg(\knows{\CM}{M})$\quad(CM message communicability, CMMC)}
		\item \colorbox[gray]{0.75}{$\LIiPded\neg\neg((\Iproves{M}{\phi}{\CM}{})\lequiv\phi)$\quad(possible TMM, PTMM)}
		\item $\LIiPded\neg\phi\limp\neg(\Iproofdiamond{M}{\phi}{\CM}{})$\\
				(falsehood implies falsehood non-disprovability, FIFND)
		\item \colorbox[gray]{0.75}{\begin{tabular}[t]{@{}l@{}}
				$\LIiPded\neg(\Iproofdiamond{M}{\phi}{\CM}{})\lequiv\neg(\Iproves{M}{\phi}{\CM}{})$\\ 
				(falsehood non-disprovability equals truth unprovability, FNDETU)\end{tabular}}
		\item $\LIiPded\neg\phi\limp\neg(\Iproves{M}{\phi}{\CM}{})$\\
				(falsehood implies truth unprovability, FITU)
		\item $\LIiPded\neg\neg\phi\limp\neg(\Iproofdiamond{M}{\neg\phi}{\CM}{})$\quad(FIFND \emph{bis})
		\item \colorbox[gray]{0.75}{$\LIiPded\neg(\Iproofdiamond{M}{\neg\phi}{\CM}{})\lequiv\neg(\Iproves{M}{\neg\phi}{\CM}{})$\quad(FNDETU \emph{bis})}
		\item $\LIiPded\neg\neg\phi\limp\neg(\Iproves{M}{\neg\phi}{\CM}{})$\quad(FITU \emph{bis})
		\item \colorbox[gray]{0.75}{\begin{tabular}[t]{@{}l@{}l@{}}
				$\LIiPded$&$(\phi\lor(\Iproves{M}{\phi}{\CM}{})\lor\Iproofdiamond{M}{\phi}{\CM}{})\limp$\\
					&$(\neg\neg\phi\land\neg(\Iproofdiamond{M}{\neg\phi}{\CM}{})\land\neg(\Iproves{M}{\neg\phi}{\CM}{}))$\\ \multicolumn{2}{@{}l}{(extended weak double-negation law, EWDN)}
				\end{tabular}}		
		\item \colorbox[gray]{0.75}{\begin{tabular}[t]{@{}l@{}}
				$\LIiPded(\neg\neg\phi\limp\phi')\limp
						((\Iproofdiamond{M}{\phi}{\CM}{})\limp\Iproves{M}{\phi'}{\CM}{})$\\
				(conditional functionality, CF)
				\end{tabular}}
		\item \colorbox[gray]{0.75}{\begin{tabular}[t]{@{}l@{}}
					$\LIiPded(\neg\neg\phi\limp\phi)\limp
						((\Iproofdiamond{M}{\phi}{\CM}{})\lequiv\Iproves{M}{\phi}{\CM}{})$\\
				(local classicality implies modal equivalence, LCIME)
				\end{tabular}}
		\item \colorbox[gray]{0.75}{\begin{tabular}[t]{@{}l@{}}
				$\LIiPded(\knows{\CM}{M}\land\Iproves{M}{\phi}{\CM}{})\limp(\neg\neg\phi\limp\phi)$\\ 
				(proof knowledge implies local classicality, PKILC)
				\end{tabular}}
		\item \colorbox[gray]{0.75}{\begin{tabular}[t]{@{}l@{}}
				$\LIiPded(\knows{\CM}{M}\land\Iproves{M}{\phi}{\CM}{})\limp((\Iproofdiamond{M}{\phi}{\CM}{})\lequiv\Iproves{M}{\phi}{\CM}{})$\\
				(proof knowledge implies modal equivalence, PKIME)
				\end{tabular}}
		\item $\LIiPded\neg(\Iproves{M}{\false}{\CM}{})$\quad(nothing can prove tautological falsehood)
		\item $\LIiPded\neg(\Iproofdiamond{M}{\false}{\CM}{})$\quad(nothing can disprove tautological truth)
		\item $\LIiPded\phi\limp\Iproofdiamond{M}{\phi}{\CM}{}$\quad(weak MM, WMM)
		\item $\set{\knows{\CM}{M}\limp\phi}\LIiPdedBis\Iproves{M}{\phi}{\CM}{}$\quad(epistemic N, EN)
		\item $\set{\knows{\CM}{M}\limp\knows{\CM}{M'}}\LIiPdedBis\Iproves{M}{\knows{\CM}{M'}}{\CM}{}$\quad(EN \emph{bis})
		\item \colorbox[gray]{0.75}{$\LIiPded\Iproves{M}{((\Iproves{M}{\phi}{\CM}{})\lequiv\phi)}{\CM}{}$\quad(ET \emph{bis} self-proof)}
		\item $\LIiPded(\Iproves{M}{\phi}{\CM}{})\limp\Iproves{M}{(\Iproves{M}{\phi}{\CM}{})}{\CM}{}$\quad(4)
		\item \colorbox[gray]{0.75}{$\LIiPded\neg(\Iproves{M}{\phi}{\CM}{})\limp\Iproves{M}{\neg(\Iproves{M}{\phi}{\CM}{})}{\CM}{}$\quad(5)}
		\item $\LIiPded((\Iproves{M}{\phi}{\CM}{})\land\Iproves{M'}{\phi'}{\CM}{})\limp
						\Iproves{\pair{M}{M'}}{(\phi\land\phi')}{\CM}{}$\\
						\quad(proof conjunctions)
		\item $\LIiPded((\Iproves{M}{\phi}{\CM}{})\land\Iproves{M}{\phi'}{\CM}{})\lequiv
						\Iproves{M}{(\phi\land\phi')}{\CM}{}$
						\quad(proof conjunctions \emph{bis})						
		\item $\LIiPded((\Iproves{M}{\phi}{\CM}{})\lor\Iproves{M'}{\phi'}{\CM}{})\limp
						\Iproves{\pair{M}{M'}}{(\phi\lor\phi')}{\CM}{}$\quad(proof disjunctions)
		\item $\LIiPded((\Iproves{M}{\phi}{\CM}{})\lor\Iproves{M}{\phi'}{\CM}{})\limp
						\Iproves{M}{(\phi\lor\phi')}{\CM}{}$\quad(proof disjunctions \emph{bis})
		\item $\LIiPded(\Iproves{M}{(\Iproves{M}{\phi}{\CM}{})}{\CM}{})\lequiv
				\Iproves{M}{\phi}{\CM}{}$\quad(modal idempotency, MI)
		\item $\LIiPded(\Iproofdiamond{M}{(\Iproofdiamond{M}{\phi}{\CM}{})}{\CM}{})\lequiv
				\Iproofdiamond{M}{\phi}{\CM}{}$\quad(MI \emph{bis})
		\item \colorbox[gray]{0.75}{
				\begin{tabular}[t]{@{}l@{}}
					$\LIiPded(\phi\lor(\Iproves{M}{\phi}{\CM}{})\lor\Iproofdiamond{M}{\phi}{\CM}{})\limp\Iproofdiamond{M}{(\Iproves{M}{\phi}{\CM}{})}{\CM}{}$\\
					(nested MM, NMM)
				\end{tabular}}
		\item \colorbox[gray]{0.75}{$\LIiPded(\Iproves{M}{(\Iproofdiamond{M}{\phi}{\CM}{})}{\CM}{})\limp\Iproofdiamond{M}{(\Iproves{M}{\phi}{\CM}{})}{\CM}{}$\quad(modal swap, MS)}
		\item \colorbox[gray]{0.75}{
				\begin{tabular}[t]{@{}l@{}}
					$\LIiPded\knows{\CM}{M}\limp((\Iproves{M}{(\phi\lor\phi')}{\CM}{})\limp((\Iproves{M}{\phi}{\CM}{})\lor\Iproves{M}{\phi'}{\CM}{}))$\\
					(epistemic internalised disjunction property, EIDP)
				\end{tabular}}
		\item \colorbox[gray]{0.75}{
				\begin{tabular}[t]{@{}l@{}}
					$\LIiPded\Iproves{M}{((\Iproves{M}{(\phi\lor\phi')}{\CM}{})\limp((\Iproves{M}{\phi}{\CM}{})\lor\Iproves{M}{\phi'}{\CM}{}))}{\CM}{}$\\
					(IDP self-proof)
				\end{tabular}}
		\item $\LIiPded(\Iproves{M}{\phi}{\CM}{})\lequiv\Iproves{M}{(\knows{\CM}{M}\land\phi)}{\CM}{}$
				\quad(epistemic idempotency, EI)
		\item $\LIiPded(\Iproves{M}{\phi}{\CM}{})\lequiv\Iproves{M}{(\knows{\CM}{M}\land\Iproves{M}{\phi}{\CM}{})}{\CM}{}$
				\quad(EI \emph{bis})
		\item \colorbox[gray]{0.75}{\begin{tabular}[t]{@{}l@{}}
				$\LIiPded(\Iproofdiamond{M}{(\phi\lor\phi')}{\CM}{})\lequiv
				((\Iproofdiamond{M}{\phi}{\CM}{})\lor\Iproofdiamond{M}{\phi'}{\CM}{})$\\ 
				(Plotkin-Stirling 4, PS4)
				\end{tabular}} 
		\item \colorbox[gray]{0.75}{\begin{tabular}[t]{@{}l@{}}
				$\LIiPded(\Iproves{M}{(\phi\limp\phi')}{\CM}{})\limp
				((\Iproofdiamond{M}{\phi}{\CM}{})\limp\Iproofdiamond{M}{\phi'}{\CM}{})$\\ 
					(Plotkin-Stirling 2, PS2)
					\end{tabular}}
		\item \colorbox[gray]{0.75}{\begin{tabular}[t]{@{}l@{}}
				$\LIiPded((\Iproofdiamond{M}{\phi}{\CM}{})\limp\Iproves{M}{\phi'}{\CM}{})\limp
				\Iproves{M}{(\phi\limp\phi')}{\CM}{}$\\ 
					(Plotkin-Stirling 5, PS5)
				\end{tabular}}
		\item \colorbox[gray]{0.75}{$\LIiPded(\neg\neg\phi\limp\phi')\limp
				\Iproves{M}{(\phi\limp\phi')}{\CM}{}$\quad(CFPS5)}
		\item \colorbox[gray]{0.75}{\begin{tabular}[t]{@{}l@{}l@{}}
				$\LIiPded$&$((\neg\neg\phi\limp\phi')\land(\neg\neg\phi'\limp\phi))\limp$\\  
						&$((\Iproves{M}{(\phi\lor\phi')}{\CM}{})\limp((\Iproves{M}{\phi}{\CM}{})\lor\Iproves{M}{\phi'}{\CM}{}))$\\
				\multicolumn{2}{@{}l}{(conditional IDP, CIDP)}
				\end{tabular}}
		\item \colorbox[gray]{0.75}{$\LIiPded(\phi\lor(\Iproves{M}{\phi}{\CM}{})\lor\Iproofdiamond{M}{\phi}{\CM}{})\limp\Iproves{M}{(\Iproofdiamond{M}{\phi}{\CM}{})}{\CM}{}$\quad(NMM \emph{bis})}
		\item \colorbox[gray]{0.75}{$\LIiPded(\Iproofdiamond{M}{(\Iproves{M}{\phi}{\CM}{})}{\CM}{})\limp
					\Iproves{M}{(\Iproofdiamond{M}{\phi}{\CM}{})}{\CM}{}$\quad(MS \emph{bis})}
		\item \colorbox[gray]{0.75}{\begin{tabular}[t]{@{}l@{}}
				$\LIiPded(\Iproves{M}{(\Iproofdiamond{M}{\phi}{\CM}{})}{\CM}{})\lequiv
							\Iproofdiamond{M}{(\Iproves{M}{\phi}{\CM}{})}{\CM}{}$\\
								(modal commutativity, MC)
								\end{tabular}}
		\item \colorbox[gray]{0.75}{
					\begin{tabular}[t]{@{}l}
						$\LIiPded(\Iproofdiamond{M}{(\Iproves{M}{\phi}{\CM}{})}{\CM}{})\lequiv\Iproofdiamond{M}{\phi}{\CM}{}$\\
							(mixed modal idempotency, MMI)
							\end{tabular}}
		\item \colorbox[gray]{0.75}{
				$\LIiPded(\Iproves{M}{(\Iproofdiamond{M}{\phi}{\CM}{})}{\CM}{})\lequiv
							\Iproofdiamond{M}{\phi}{\CM}{}$\quad(MMI \emph{bis})}
		\item $\LIiPded(\Iproves{M}{\neg\neg\phi}{\CM}{})\limp
							\Iproofdiamond{M}{\phi}{\CM}{}$\quad(double-negation absorption, DNA)			
		\item $\LIiPded(\Iproves{M}{\neg\neg\phi}{\CM}{})\limp
							\neg\neg(\Iproves{M}{\phi}{\CM}{})$\quad(double-negation extrusion, DNE)
		\item \colorbox[gray]{0.75}{$\LIiPded\neg(\Iproofdiamond{M}{\phi}{\CM}{})\limp
							\Iproofdiamond{M}{\neg\phi}{\CM}{}$\quad(weak negation completeness, WNC)}					\item \colorbox[gray]{0.75}{$\LIiPded\neg(\Iproves{M}{\phi}{\CM}{})\limp
							\Iproofdiamond{M}{\neg\phi}{\CM}{}$\quad(WNC \emph{bis})}
		\item \colorbox[gray]{0.75}{$\boxed{\LIiPded(\Iproofdiamond{M}{\phi}{\CM}{})\lequiv
										\neg\neg(\Iproves{M}{\phi}{\CM}{})}$\quad(modal double negation, MDN)}
		\item \colorbox[gray]{0.75}{$\LIiPded((\Iproves{M}{\phi}{\CM}{})\land
							\Iproofdiamond{M}{(\phi\limp\phi')}{\CM}{})\limp
								\Iproofdiamond{M}{\phi}{\CM}{}$\quad(Wiv)}
		\item \colorbox[gray]{0.75}{$\LIiPded(\Iproofdiamond{M}{(\phi\land\neg\phi')}{\CM}{})\limp
								(\phi\land\neg\phi')$\quad(Wv)}
\end{enumerate}
\end{theorem}
\begin{proof}
	See Appendix~\ref{appendix:Proofs:LogicalLaws}.
\end{proof}
\noindent
The following remark flags a non-trivial insight, also explicated 
	in Section~\ref{section:Semantically}.
\begin{remark}[Intuitionistic truths]\label{remark:IntuitionisticTruths}
	Theorem~\ref{theorem:SomeUsefulDeducibleLogicalLaws}.8  
		means that 
			in interactive settings, 
				the intuitionistic truths are those of the communication medium.
\end{remark}
\noindent
The reader is invited to 
	compare Theorem~\ref{theorem:SomeUsefulDeducibleLogicalLaws}.8 to its global counterpart 
			Theorem~\ref{theorem:SomeUsefulDeducibleLogicalLaws}.27,
				whose analog also holds in LiP but was not stated there.
\begin{corollary}[Modal-depth result]\label{corollary:ModalDepth}\ 
	Let $\heartsuit_{1}\cdots\heartsuit_{n}\phi\in\pFormulas$ such that 
		for all $1\leq i\leq n$, 
			$\heartsuit_{i}\in\set{\text{`}\Iproves{M}{}{\CM}{}\negthickspace\text{'}\,,\text{`}\Iproofdiamond{M}{}{\CM}{}\negthickspace\text{'}\;}$.
	Then, if the prefix `\,$\heartsuit_{1}\cdots\heartsuit_{n}$\negthinspace' contains 
		\begin{enumerate}
			\item only occurrences of the `$\Iproves{M}{}{\CM}{}$\negthickspace'-modality 
					then $$\LIiPded(\heartsuit_{1}\cdots\heartsuit_{n}\phi)\lequiv
									\Iproves{M}{\phi}{\CM}{}\,;$$
			\item at least one occurrence of the `$\Iproofdiamond{M}{}{\CM}{}$\negthickspace'-modality 
					then $$\LIiPded(\heartsuit_{1}\cdots\heartsuit_{n}\phi)\lequiv
									\Iproofdiamond{M}{\phi}{\CM}{}\,.$$
		\end{enumerate}
\end{corollary}
\begin{proof}
	For 1, apply MI.
	For 2,
		apply MI, MI \emph{bis}, MMI, and MMI \emph{bis}.
\end{proof}

\begin{corollary}[Intuitionistic Modal Logic]\label{corollary:IK}
	LIiP is a refinement (due to its parameterised modality) and extension (due to additional laws) of 
		Fischer Servi's \cite{IK:FischerServi} and   
		Plotkin and Stirling's \cite{IK:PlotkinStirling} 
	intuitionistic modal logic IK.
	
	Similarly is 
		LIiP a refinement and extension of 	
			the propositional fragment of 
				Wijesekera's system of first-order constructive modal logic \cite[Section~1.5]{ConstructiveModalLogicI}.
\end{corollary}
\begin{proof}
	In fact, 
		Plotkin and Stirling's axiomatisation of IK, which 
			is equivalent to Fischer Servi's, 
				consists of the axioms of IL and the laws K, MP, N as well as the axiom analogs of 
					Theorem~\ref{theorem:SomeUsefulDeducibleLogicalLaws}.25 and 
					\ref{theorem:SomeUsefulDeducibleLogicalLaws}.44--\ref{theorem:SomeUsefulDeducibleLogicalLaws}.46; and
		the propositional fragment of Wijesekera's system consists of the axioms of IL and the laws K, MP, N as well as the axiom analogs of 
					Theorem~\ref{theorem:SomeUsefulDeducibleLogicalLaws}.45, 
					\ref{theorem:SomeUsefulDeducibleLogicalLaws}.59 and 
					\ref{theorem:SomeUsefulDeducibleLogicalLaws}.60.
\end{proof}

The following corollary asserts that 
	our disjunctive proof modality is also 
		an \emph{explicit refinement} of the standard (implicit) belief modality \cite{MultiAgents}.
\begin{corollary}[Disjunctive Explicit Belief]\label{corollary:DEB}\ 
	`$\Iproves{M}{\cdot}{\CM}{}$' is 
				a \emph{disjunctive KD45-modality} of explicit agent belief,
	where 
		$M$ represents the explicit evidence term that can justify the agent $\CM$'s belief.
	Additionally,
		the communication medium $\CM$ is a truth-believing agent in the sense that 
			$\LIiPded\phi\limp\Iproves{M}{\phi}{\CM}{}$.
			
\end{corollary}
\begin{proof}
	Consider that 
		`$\Iproves{M}{\cdot}{\CM}{}$' satisfies 
			the K-law (\cf Definition~\ref{definition:AxiomsRules}), 
			the D-law (called ID in Definition~\ref{definition:AxiomsRules}),
			the 4-law (\cf Theorem~\ref{theorem:SomeUsefulDeducibleLogicalLaws}.30),
			the 5-law (\cf Theorem~\ref{theorem:SomeUsefulDeducibleLogicalLaws}.31), 
			IDP self-proof (\cf Theorem~\ref{theorem:SomeUsefulDeducibleLogicalLaws}.41), 
			the MM-law (\cf Definition~\ref{definition:AxiomsRules}), and
			the N-law (\cf Theorem~\ref{theorem:SomeUsefulDeducibleLogicalLaws}.0).
\end{proof}
Thanks to Theorem~\ref{theorem:SomeUsefulDeducibleLogicalLaws}.10, 
	$\knows{\CM}{M}$ is a sufficient condition for `$\Iproves{M}{\cdot}{\CM}{}$' to
		behave like a standard S5-modality of \emph{perfect knowledge}\label{page:PerfectKnowledge}   
		(in a technical sense) \cite{MultiAgents,Epistemic_Logic,EpistemicLogicFiveQuestions}, which 
			in addition to being a KD45-modality not only obeys the D-law but also the stronger T-law (knowledge, not only belief) and the MM-law (\emph{perfect} knowledge):  
				$$\LIiPded\knows{\CM}{M}\limp((\Iproves{M}{\phi}{\CM}{})\lequiv\phi).$$
\begin{remark}[Perfect knowledge]\label{remark:PerfectKnowledge}
In interactive settings,
	only the communication medium $\CM$, 
		through which all messages have to pass, 
			can attain perfect knowledge 
				(the other agents having only 
					partial visibility of the network, 
						\cf Definition~\ref{definition:SemanticIngredients}).
However note that 
	LIiP being propositionally modal-intuitionistic,  
	this perfect knowledge is of propositional invariants of 
		the communication network only (\cf Page~\pageref{page:invariants}).  
So the epistemic perfection of the communication medium is only within a certain \emph{grain} (propositional) and \emph{scope} (invariants).
\end{remark}

In the following corollary, 
	we construct also 
		a disjunctive explicit refinement of (implicit) S4-provability.
\begin{corollary}[Disjunctive Explicit Provability]\label{corollary:DEP}
	`\,$\knows{\CM}{M}\land\Iproves{M}{\cdot}{\CM}{}$' is 
		a \emph{disjunctive S4-modality} of explicit agent provability, where 
			$M$ represents the explicit evidence term that does justify agent $\CM$'s knowledge.
\end{corollary}
\begin{proof}
	By Corollary~\ref{corollary:DEB}, 
		Theorem~\ref{theorem:SomeUsefulDeducibleLogicalLaws}.10, and
		Theorem~\ref{theorem:SomeUsefulDeducibleLogicalLaws}.40:
			The T-law  
				$\LDiiPded(\knows{\CM}{M}\land\Iproves{M}{\phi}{\CM}{})\limp\phi$ for 
					the modality `$\knows{\CM}{M}\land\Iproves{M}{\cdot}{\CM}{}$' can be recognised by 
						inspecting Theorem~\ref{theorem:SomeUsefulDeducibleLogicalLaws}.10, and 
			the disjunctivity $\LDiiPded(\knows{\CM}{M}\land\Iproves{M}{(\phi\lor\phi')}{\CM}{})\limp 
				((\knows{\CM}{M}\land\Iproves{M}{\phi}{\CM}{})\lor(\knows{\CM}{M}\land\Iproves{M}{\phi'}{\CM}{}))$ 
		by inspecting Theorem~\ref{theorem:SomeUsefulDeducibleLogicalLaws}.40.
\end{proof}

\subsection{Semantically}\label{section:Semantically}
We continue to present 
	the concretely constructed semantics as well as 
	the standard abstract semantic interface for LIiP, and 
prove the axiomatic adequacy of the proof system with respect to this interface.
The core ingredient of the concrete semantics of LIiP are so-called \emph{input histories,} which
	were introduced in \cite{LDiiP,KramerIMLA2013} and 
	could also be used in an even more concrete semantics of LiP. 
Input histories 
	are finite words of input events and 
	serve as concrete states $s\in\states$ in 
		the state space $\states$, on which 
			the concrete and abstract accessibility relation 
				${\pAccess{M}{\CM}{}}\subseteq\states\times\states$ and
				${\access{M}{\CM}{}}\subseteq\states\times\states$ for LIiP is defined, respectively.
The reader of \cite{LiP,KramerICLA2013} will recognise 
	similar but simpler definitions here;  
		the one of ${\pAccess{M}{\CM}{}}$ could be even simpler (\cf Fact~\ref{fact:identities}.2), but 
			is as now in order to allow for 
				a simpler, pattern-matching comparison with 
					the corresponding one in \cite{LiP,KramerICLA2013}.
(We wanted to show how to produce LIiP from LiP.)

\subsubsection{Concretely}\label{section:Concretely}

\begin{definition}[Semantic ingredients]\label{definition:SemanticIngredients}
For the set-theoretically constructive, model-theoretic study of LIiP let 
\begin{itemize}
	\item $\states\ni s\bnfeq\mathtt{0}\bnfor \gscc{a}{M}(s)$ designate the concrete state space $\states$ of
			\emph{input histories} $s$, where 
				$\mathtt{0}$ designates the empty input history  
					(\ie a zero data point, \eg an initial state) and 
				$\gscc{a}{M}$ can be read as ``agent $a$ receives message $M$''  
					(\eg from some other agent acting as an oracle for $a$); and 
					$\star:(\states\times\states)\to\states$ monoidal concatenation on $\states$
						(with neutral element $\mathtt{0}$);
	\item $\pi_{a}:\states\rightarrow\states$ designate (local) \emph{state projection on $a$'s view} such that 
			\begin{align*}
				\pi_{a}(\mathtt{0}) &\defeq \mathtt{0}\\
				\pi_{a}(\gscc{b}{M}(s)) &\defeq 
					\begin{cases}
						\gscc{b}{M}(\pi_{a}(s)) & \text{if $a\in\set{b,\CM}$, and}\\
						\pi_{a}(s) & \text{otherwise;}
					\end{cases}
			\end{align*}
			(The communication medium $\CM$ sees any agent's $b$ [including its own] input events, \ie
				$\CM$ has a global view on the current global state $s$.)
	\item $\msgs{}:\states\rightarrow\powerset{\messages}$ designate \emph{raw-data extraction} such that 
			\begin{align*}
				\msgs{}(\mathtt{0}) &\defeq \emptyset\\
				\msgs{}(\gscc{a}{M}(s)) &\defeq \msgs{}(s)\cup\set{M}\,;
			\end{align*}
	\item $\msgs{a}\defeq\msgs{}\circ\pi_{a}$ designate (local) \emph{raw-data extraction by $a\,;$}
		\item $\clo{a}{s}:\powerset{\messages}\rightarrow\powerset{\messages}$ designate a \emph{data-mining operator} such that \label{page:DataMining}
			$\clo{a}{s}(\data)\defeq\clo{a}{}(\msgs{a}(s)\cup\data)\defeq\bigcup_{n\in\mathbb{N}}\clo{a}{n}(\msgs{a}(s)\cup\data)$, where for all $\data\subseteq\messages:$
				\begin{eqnarray*}
					\clo{a}{0}(\data) &\defeq& \set{a}\cup\data\\
					\clo{a}{n+1}(\data) &\defeq& 
						\begin{array}[t]{@{}l@{}}
							\clo{a}{n}(\data)\ \cup\\
							\setst{\pair{M}{M'}}{\set{M,M'}\subseteq\clo{a}{n}(\data)}\cup\quad\text{(pairing)}\\
							\setst{M, M'}{\pair{M}{M'}\in\clo{a}{n}(\data)}\quad\text{(unpairing)}
						\end{array}
				\end{eqnarray*}
				($\clo{a}{s}(\emptyset)$ can be viewed as $a$'s \emph{individual-knowledge base} in $s$. 
					For application-specific terms such as signing and encryption, 
					we would have to add here the closure conditions corresponding to their characteristic term axioms.)
	\item ${\sqsubseteq_{a}}\subseteq\states\times\states$ designate 
		the (local) \emph{state \textbf{pre-}order of $a$} such that 
			for all $s,s'\in\states$, 
				$s\sqsubseteq_{a} s'$ :iff 
					there is $s''\in\states$ such that 
						$\pi_{a}(s)\star\pi_{a}(s'')=\pi_{a}(s')\,;$
	\item ${\sqsubseteq}\defeq{\sqsubseteq_{\CM}}$ designate the (global) \emph{state \textbf{partial} order} 
			serving as the concrete accessibility relation in the Kripke-semantics for the I-fragment of LIiP;
	
			($\sqsubseteq$ is partial thanks to 
				$\CM$ seeing any agent's input events.)
	\item ${\equiv_{a}}\defeq{\sqsubseteq_{a}}\cap(\sqsubseteq_{a})^{-1}$ designate the (local) \emph{state equivalence of $a$;}
	\item ${\pAccess{M}{\CM}{}}\subseteq\states\times\states$ designate 
		the \emph{\textbf{concretely constructed} accessibility relation}---short, \emph{\textbf{concrete} accessibility}---for LIiP such that for all $s,s'\in\states$,  
			\begin{eqnarray*}
				s\pAccess{M}{\CM}{}s' &\defiff&  
								s'\in\hspace{-3.25ex} 
					\bigcup_{\scriptsize 
					\begin{array}{@{}c@{}}
						\text{$s\sqsubseteq_{\CM}{}\tilde{s}$ and
						}\\[0.5\jot] 
						M\in\clo{\CM}{\tilde{s}}(\emptyset) 
					\end{array}
					}\hspace{-3.25ex}
					[\tilde{s}]_{\indist{\CM}{}{}}\\
				&\text{(iff}& 
					\text{there is $\tilde{s}\in\states$ \st 
							$s\sqsubseteq_{\CM}{}\tilde{s}$ and
							$M\in\clo{\CM}{\tilde{s}}(\emptyset)$ and
							$\indist{\CM}{\tilde{s}}{s'}$).}
			\end{eqnarray*}
\end{itemize}
\end{definition}

Note that
	the data-mining operator
		$\clo{a}{}:\powerset{\messages}\rightarrow\powerset{\messages}$ is a compact closure operator, which
		induces a \emph{data-derivation relation} ${\derives{a}{}{}}\subseteq\powerset{\messages}\times\messages$ such that $\derives{a}{M}{\data}$ :iff $M\in\clo{a}{}(\data)$, which  
	(1) has the compactness and (2) the cut property,  
	(3) is decidable in deterministic polynomial time in the size of $\data$ and $M$, and
	(4) induces a Scott information system of information tokens $M$ \cite{LiP}.
\begin{fact}\label{fact:identities}\ 
	\begin{enumerate}
		\item ${\equiv_{\CM}}=\mathrm{Id}_{\states}$
		\item $s\pAccess{M}{\CM}{}s'$ if and only if 
							($s\sqsubseteq s'$ and
							 $M\in\clo{\CM}{s'}(\emptyset)$)
		\item ${\pAccess{\CM}{\CM}{}}={\sqsubseteq}$
	\end{enumerate}
\end{fact}
\begin{proof}
	By inspection of definitions.
\end{proof}
\noindent 
Fact~\ref{fact:identities}.1 is important, because
	thanks to it 
		the communication medium $\CM$ can have perfect knowledge, 
			as asserted on Page~\pageref{page:PerfectKnowledge};
and Fact~\ref{fact:identities}.3 is, because 
	thanks to it 
		the partial order $\sqsubseteq$ for the Kripke-semantics of LIiP's intuitionistic connectives is absorbed as 
		a mere instance $\pAccess{\CM}{\CM}{}$ of the accessibility relation $\pAccess{M}{\CM}{}$ for the Kripke-semantics of LIiP's proof modality, as announced on Page~\pageref{page:PartialOrder}.
Fact~\ref{fact:identities}.2 captures the intuition of the concrete accessibility of LIiP.
Spelled out, 
	this intuition is that 
		$\CM$ can access 
			an input history $s'$ from 
			the current input history $s$ with respect to a piece of data $M$ in question 
				if and only if 
					$s'$ is an extension of $s$ such that 
						$M$ is in $\CM$'s individual-knowledge base at $s'$.
A simple example is that 
	$\mathtt{0}\pAccess{\CM}{M}{}\gscc{a}{M}(\mathtt{0})$, because
		$\mathtt{0}\sqsubseteq\gscc{a}{M}(\mathtt{0})$ and
		$M\in\clo{\CM}{\gscc{a}{M}(\mathtt{0})}(\emptyset)$.
A slightly more complex example is that 
	$\mathtt{0}\pAccess{\CM}{M'}{}\gscc{b}{\pair{M}{M'}}(\gscc{a}{M}(\mathtt{0}))$, because
		$\mathtt{0}\sqsubseteq\gscc{b}{\pair{M}{M'}}(\gscc{a}{M}(\mathtt{0}))$ and
		$M'\in\clo{\CM}{\gscc{b}{\pair{M}{M'}}(\gscc{a}{M}(\mathtt{0}))}(\emptyset)$.

We need the following auxiliary definition for 
	the proposition following it.
\begin{definition}[Message pre-ordering \cite{LiP,KramerICLA2013}]\label{definition:MessagePreordering}\ 
	\begin{itemize}
		\item $M\sqsubseteq_{a}^{s}M'$ :iff 
			if $M\in\clo{a}{s}(\emptyset)$ then $M'\in\clo{a}{s}(\emptyset)$
		\item $M\sqsubseteq_{a}M'$ :iff for all $s\in\states$, $M\sqsubseteq_{a}^{s}M'$
	\end{itemize}
\end{definition}
\noindent
Notice the definitional overloading of the notation $\sqsubseteq_{a}$, \ie  
	once as ${\sqsubseteq_{a}}\subseteq\states\times\states$ in Definition~\ref{definition:SemanticIngredients} and once as ${\sqsubseteq_{a}}\subseteq\messages\times\messages$ in the previous Definition~\ref{definition:MessagePreordering}.

\begin{proposition}[Concrete accessibility]\label{proposition:ConcreteAccessibility}\  
\begin{enumerate}
	\item If $s\pAccess{M}{\CM}{}s'$ then $M\in\clo{\CM}{s'}(\emptyset)$\quad(epistemic image)
	\item If $M\in\clo{\CM}{s}(\emptyset)$ then 
			$s\pAccess{M}{\CM}{}s$\quad(conditional reflexivity)
	\item there is $s'\in\states$ such that $s\pAccess{M}{\CM}{}s'$\quad(seriality)
	\item ${\pAccess{M}{\CM}{}}\subseteq{\pAccess{\CM}{\CM}{}}={\sqsubseteq}$
				\quad(MIAR-inclusion)
	\item $({\pAccess{\CM}{\CM}{}}\circ{\pAccess{M}{\CM}{}})\subseteq{\pAccess{M}{\CM}{}}$
				\quad(special transitivity)
	\item If $M\sqsubseteq_{\CM}M'$ 
			then ${\pAccess{M}{\CM}{}}\subseteq{\pAccess{M'}{\CM}{}}$\quad
				(proof monotonicity)
\end{enumerate}
\end{proposition}
\begin{proof}
	For 1, inspect Fact~\ref{fact:identities}.2.
	For 2, inspect Fact~\ref{fact:identities}.2 and \ref{fact:identities}.3 and 
		the definition of $\sqsubseteq$ (to see that $\sqsubseteq$ is reflexive).
	For 3,
		inspect Fact~\ref{fact:identities}.2 and 
		let $s\in\states$.
		Then choose $s'=\gscc{\CM}{M}(s)\in\states$, which implies that 
			$s\sqsubseteq s'$ and
			$M\in\clo{\CM}{s'}(\emptyset)$.
	For 4, inspect Fact~\ref{fact:identities}.2 and \ref{fact:identities}.3 and 
		the definitional fact that $\CM\in\clo{\CM}{s'}(\emptyset)$.
	For 5,
		let $s,s'\in\states$ and
		suppose that $s\mathrel{({\pAccess{\CM}{\CM}{}}\circ{\pAccess{M}{\CM}{}})}s'$.
	That is,
		there is $s''\in\states$ such that 
			$s\pAccess{\CM}{\CM}{}s''$ and $s''\pAccess{M}{\CM}{}s'$.
	Hence, 
		$s\sqsubseteq s''$ and $\CM\in\clo{\CM}{s''}(\emptyset)$ as well as
		$s''\sqsubseteq s'$ and $M\in\clo{\CM}{s'}(\emptyset)$, by 
			Fact~\ref{fact:identities}.2.
	Hence 
		$s\sqsubseteq s'$ by the transitivity of $\sqsubseteq$.
	Hence $s\pAccess{M}{\CM}{}s'$ by Fact~\ref{fact:identities}.2.
	For 6,
		let $M,M'\in\messages$ and 
		suppose that $M\sqsubseteq_{\CM}M'$.
		Further, 
			let $s,s'\in\states$ and 
			suppose that $s\pAccess{M}{\CM}{}s'$.
		Hence 
			$M\sqsubseteq_{\CM}^{s'}M'$, and also 
			$s\sqsubseteq s'$ and $M\in\clo{\CM}{s'}(\emptyset)$ by Fact~\ref{fact:identities}.2.
		Hence $M'\in\clo{\CM}{s'}(\emptyset)$.
		Hence $s\pAccess{M'}{\CM}{}s'$ by Fact~\ref{fact:identities}.2.
\end{proof}
\noindent
Note that ``MIAR'' stands for ``modal-intuitionistic-accessibility-relation.''

\subsubsection{Abstractly}\label{section:Abstractly}

\begin{definition}[Kripke-model]\label{definition:KripkeModel}
We define the \emph{satisfaction relation} $\models$ for 
		LIiP in Table~\ref{table:SatisfactionRelation}, 
	\begin{table}[t]
	\centering
	\caption{Satisfaction relation}
	\smallskip
	\fbox{$\begin{array}{@{}rcl@{}}
		(\aModalFrame, \mathcal{V}), s\models P &\text{:iff}& s\in\mathcal{V}(P)\\[\jot]
		(\aModalFrame, \mathcal{V}), s\models\phi\lor\phi' &\text{:iff}& \text{$(\aModalFrame, \mathcal{V}), s\models\phi$ or $(\aModalFrame, \mathcal{V}), s\models\phi'$}\\[\jot]
		(\aModalFrame, \mathcal{V}), s\models\phi\land\phi' &\text{:iff}& \text{$(\aModalFrame, \mathcal{V}), s\models\phi$ and $(\aModalFrame, \mathcal{V}), s\models\phi'$}\\[\jot]
		(\aModalFrame, \mathcal{V}), s\models\neg\phi &\text{:iff}& 
			\text{for all $s'\in\states$, 
				if $s\sqsubseteq s'$ 
				then not $(\aModalFrame, \mathcal{V}), s'\models\phi$}\\[\jot]
		(\aModalFrame, \mathcal{V}), s\models\phi\limp\phi' &\text{:iff}& 
			\text{for all $s'\in\states$, 
				if $s\sqsubseteq s'$ 
				then 
					\begin{tabular}{@{}r@{}l@{}}
						(&not $(\aModalFrame, \mathcal{V}), s'\models\phi$\\  
						 &or $(\aModalFrame, \mathcal{V}), s'\models\phi'$)
					\end{tabular}}\\[\jot]
		(\aModalFrame, \mathcal{V}), s\models\Iproves{M}{\phi}{\CM}{} &\text{:iff}& 
			\begin{array}[t]{@{}l@{}}
				\text{for all $s'\in\states$, }
			 	\text{if $s\access{M}{\CM}{}s'$ then $(\aModalFrame, \mathcal{V}), s'\models\phi$}
			\end{array}
	\end{array}$}
	\label{table:SatisfactionRelation}
	\end{table}
where 
	\begin{itemize}
		\item $\mathcal{V}:\mathcal{P}\rightarrow\powerset{\states}$ designates a usual \emph{valuation function,} yet
			\begin{itemize}
				\item \emph{partially predefined} such that for all $a\in\agents$ and $M\in\messages$,
							$$\mathcal{V}(\knows{a}{M})\defeq\setst{s\in\states}{M\in\clo{a}{s}(\emptyset)}\,;$$
				
							(If agents are Turing-machines 
							then $a$ knowing $M$ can be understood as $a$ being able to parse $M$ on its tape.)
				\item \emph{constrained} such that for all $s,s'\in\states$,
						$$\text{if $s\in\mathcal{V}(P)$ and $s\sqsubseteq s'$ then $s'\in\mathcal{V}(P)\,;$}$$
						(following Kripke's semantics for IL)
			\end{itemize}
		\item $\aModalFrame\defeq(\states,\sqsubseteq,\set{\access{M}{\CM}{}}_{M\in\messages})$
			designates an intuitionistic modal \emph{frame} for LIiP with 
		a usual partial order ${\sqsubseteq}\subseteq\states\times\states$ for the intuitionistic part as well as 
		an \emph{\textbf{abstractly constrained} accessibility relation}---short, \emph{\textbf{abstract}  accessibility}---${\access{M}{\CM}{}}\subseteq\states\times\states$ for 
			LIiP such that---the \emph{semantic interface:}\label{page:AbstractProofAccessibility} 
			\begin{itemize}
				\item If $s\access{M}{\CM}{}s'$ then $M\in\clo{\CM}{s'}(\emptyset)$
				\item If $M\in\clo{\CM}{s}(\emptyset)$ then 
						$s\access{M}{\CM}{}s$
				\item there is $s'\in\states$ such that $s\access{M}{\CM}{}s'$
				\item ${\access{M}{\CM}{}}\subseteq{\access{\CM}{\CM}{}}={\sqsubseteq}$
				\item $({\access{\CM}{\CM}{}}\circ{\access{M}{\CM}{}})\subseteq{\access{M}{\CM}{}}$
				\item If $M\sqsubseteq_{\CM}M'$ 
						then ${\access{M}{\CM}{}}\subseteq{\access{M'}{\CM}{}}$
			\end{itemize}
	\item $(\aModalFrame,\mathcal{V})$ designates an intuitionistic modal \emph{model} for LIiP.
	\end{itemize}
\end{definition}
\noindent
Looking back, 
	we recognise that 
		Proposition~\ref{proposition:ConcreteAccessibility} 
	actually establishes the important fact that
		our concrete accessibility $\pAccess{M}{\CM}{}$ in 
		Definition~\ref{definition:SemanticIngredients} realises 
			all the properties stipulated by  
		our abstract accessibility $\access{M}{\CM}{}$ in Definition~\ref{definition:KripkeModel};
		we say that 
		$$\text{$\pAccess{M}{\CM}{}$ \emph{exemplifies} (or \emph{realises}) $\access{M}{\CM}{}$\,.}$$
Further, observe that 
	LIiP (like LiP) has a Herbrand-style semantics, \ie
			logical constants (agent names) and 
			functional symbols (pairing) are self-interpreted rather than 
		interpreted in terms of (other, semantic) constants and functions.
This simplifying design choice spares our framework from 
	the additional complexity that would arise from term-variable assignments \cite{FOModalLogic}, which in turn
		keeps our models propositionally modal.
Our choice is admissible because our individuals (messages) are finite.
(Infinitely long ``messages'' are non-messages; they can never be completely received, \eg
	transmitting irrational numbers as such is impossible.)

\begin{definition}[Truth \& Validity \cite{ModalLogicSemanticPerspective}]\label{definition:TruthValidity}\  
	\begin{itemize}
	\item The formula $\phi\in\pFormulas$ is \emph{true} (or \emph{satisfied}) 
		in the model $(\aModalFrame,\mathcal{V})$ at the state $s\in\states$ 
			:iff $(\aModalFrame,\mathcal{V}), s\models\phi$.
	\item The formula $\phi$ is \emph{satisfiable} in the model $(\aModalFrame,\mathcal{V})$ 
			:iff there is $s\in\states$ such that 
				$(\aModalFrame,\mathcal{V}), s\models\phi$.
	\item The formula $\phi$ is \emph{globally true} (or \emph{globally satisfied}) 
		in the model $(\aModalFrame,\mathcal{V})$, 
			written $(\aModalFrame,\mathcal{V})\models\phi$, :iff 
				for all $s\in\states$, $(\aModalFrame,\mathcal{V}),s\models\phi$.
	\item The formula $\phi$ is \emph{satisfiable}  
			:iff there is a model $(\aModalFrame,\mathcal{V})$ and a state $s\in\states$ such that 
				$(\aModalFrame,\mathcal{V}),s\models\phi$.	
	\item The formula $\phi$ is \emph{valid}, written $\models\phi$, :iff 
			for all models $(\aModalFrame,\mathcal{V})$, $(\aModalFrame,\mathcal{V})\models\phi$.
	\end{itemize}
\end{definition}
\noindent
The following lemma is a \emph{passage oblig\'{e}} in the construction of 
	a Kripke-semantics for any intuitionistic logic, modal or not, and thus also for ours.
\begin{lemma}[Monotonicity Lemma]
	For all LIiP-models $(\aModalFrame,\mathcal{V})$, $s,s'\in\states$, and $\phi\in\pFormulas$, 
		if $s\sqsubseteq s'$ and $(\aModalFrame,\mathcal{V}),s\models\phi$ 
		then $(\aModalFrame,\mathcal{V}),s'\models\phi$.
\end{lemma}
\begin{proof} 
	Let $(\aModalFrame,\mathcal{V})$ designate an arbitrary LIiP-model and let $s,s'\in\states$.
	Then let us proceed by induction on the structure of $\phi\in\pFormulas:$
	\begin{itemize}
		\item Base case ($\phi\defeq P$ for an arbitrary $P\in\mathcal{P}\subseteq\pFormulas$).
				Suppose that $s\sqsubseteq s'$ and $(\aModalFrame,\mathcal{V}),s\models P$.
				Thus $s\in\mathcal{V}(P)$, by definition.
				Hence $s'\in\mathcal{V}(P)$, by the definitional constraint on $\mathcal{V}$.
				Thus $(\aModalFrame,\mathcal{V}),s'\models P$, by definition.
		\item Inductive steps:
			\begin{itemize}
				\item $\phi\defeq \phi'\lor\phi''$ for arbitrary $\phi',\phi''\in\pFormulas$.
						For the sake of the induction, suppose 
							that 
								if $s\sqsubseteq s'$ and $(\aModalFrame,\mathcal{V}),s\models\phi'$
								then $(\aModalFrame,\mathcal{V}),s'\models\phi'$ and 
							that 
								if $s\sqsubseteq s'$ and $(\aModalFrame,\mathcal{V}),s\models\phi''$
								then $(\aModalFrame,\mathcal{V}),s'\models\phi''$.
						Further suppose that 
							$s\sqsubseteq s'$ and $(\aModalFrame,\mathcal{V}),s\models\phi'\lor\phi''$.
						Thus 
							$(\aModalFrame,\mathcal{V}),s\models\phi'$ or
							$(\aModalFrame,\mathcal{V}),s\models\phi''$, by definition.
						Let us proceed by disjunctive case analysis:
						\begin{itemize}
							\item Suppose that $(\aModalFrame,\mathcal{V}),s\models\phi'$.
									Hence $(\aModalFrame,\mathcal{V}),s'\models\phi'$, by induction hypothesis.
									Hence 
										$(\aModalFrame,\mathcal{V}),s'\models\phi'$ or 
										$(\aModalFrame,\mathcal{V}),s'\models\phi''$, by 
											meta-level (classical) propositional logic.
									Thus  
										$(\aModalFrame,\mathcal{V}),s'\models\phi'\lor\phi''$, by definition.
							\item Suppose that $(\aModalFrame,\mathcal{V}),s\models\phi''$, and
									proceed symmetrically to the previous case.
						\end{itemize}
				\item $\phi\defeq \phi'\land\phi''$ for arbitrary $\phi',\phi''\in\pFormulas$.
						For the sake of the induction, suppose 
							that 
								if $s\sqsubseteq s'$ and $(\aModalFrame,\mathcal{V}),s\models\phi'$
								then $(\aModalFrame,\mathcal{V}),s'\models\phi'$ and 
							that 
								if $s\sqsubseteq s'$ and $(\aModalFrame,\mathcal{V}),s\models\phi''$
								then $(\aModalFrame,\mathcal{V}),s'\models\phi''$.
						Further suppose that 
							$s\sqsubseteq s'$ and $(\aModalFrame,\mathcal{V}),s\models\phi'\land\phi''$.
						Thus 
							$(\aModalFrame,\mathcal{V}),s\models\phi'$ and 
							$(\aModalFrame,\mathcal{V}),s\models\phi''$, by definition.
						Hence
							$(\aModalFrame,\mathcal{V}),s'\models\phi'$ and 
							$(\aModalFrame,\mathcal{V}),s'\models\phi''$, by the induction hypotheses.
						Thus
							$(\aModalFrame,\mathcal{V}),s'\models\phi'\land\phi''$, by definition.
				\item $\phi\defeq \neg\phi'$ for an arbitrary $\phi'\in\pFormulas$.
						For the sake of the induction, suppose 
							that 
								if $s\sqsubseteq s'$ and $(\aModalFrame,\mathcal{V}),s\models\phi'$
								then $(\aModalFrame,\mathcal{V}),s'\models\phi'$.
						Further suppose that 
							$s\sqsubseteq s'$ and $(\aModalFrame,\mathcal{V}),s\models\neg\phi'$.
						Thus 
							for all $s''\in\states$,
								if $s\sqsubseteq s''$ then not $(\aModalFrame,\mathcal{V}),s''\models\phi'$,
									by definition.
						Now, 
							let $s''\in\states$ and
							further suppose that $s'\sqsubseteq s''$.
						Hence $s\sqsubseteq s''$, by the transitivity of $\sqsubseteq$.
						Hence not $(\aModalFrame,\mathcal{V}),s''\models\phi'$.
						Hence for all $s''\in\states$,
							if $s'\sqsubseteq s''$ then not $(\aModalFrame,\mathcal{V}),s''\models\phi'$.
						Thus $(\aModalFrame,\mathcal{V}),s'\models\neg\phi'$, by definition.
						(Observe that 
							the induction hypothesis turns out to be irrelevant for this case.)
				\item $\phi\defeq \phi'\limp\phi''$ for arbitrary $\phi',\phi''\in\pFormulas$.
						For the sake of the induction, suppose 
							that 
								if $s\sqsubseteq s'$ and $(\aModalFrame,\mathcal{V}),s\models\phi'$
								then $(\aModalFrame,\mathcal{V}),s'\models\phi'$ and 
							that 
								if $s\sqsubseteq s'$ and $(\aModalFrame,\mathcal{V}),s\models\phi''$
								then $(\aModalFrame,\mathcal{V}),s'\models\phi''$.
						Further suppose that 
							$s\sqsubseteq s'$ and $(\aModalFrame,\mathcal{V}),s\models\phi'\limp\phi''$.
						Thus 
							for all $s''\in\states$,
								if $s\sqsubseteq s''$ then 
									(not $(\aModalFrame,\mathcal{V}),s''\models\phi'$ or
									$(\aModalFrame,\mathcal{V}),s''\models\phi''$), by definition.
						Now, 
							let $s''\in\states$ and
							further suppose that $s'\sqsubseteq s''$.
						Hence $s\sqsubseteq s''$, by the transitivity of $\sqsubseteq$.
						Hence, 
							not $(\aModalFrame,\mathcal{V}),s''\models\phi'$ or
							$(\aModalFrame,\mathcal{V}),s''\models\phi''$.
						Hence for all $s''\in\states$,
							if $s'\sqsubseteq s''$ 
							then (not $(\aModalFrame,\mathcal{V}),s''\models\phi'$ or
										$(\aModalFrame,\mathcal{V}),s''\models\phi''$).
						Thus $(\aModalFrame,\mathcal{V}),s'\models\phi'\limp\phi''$, by definition.
						(Observe that 
							the induction hypotheses turn out to be irrelevant for this case.)
				\item $\phi\defeq \Iproves{M}{\phi'}{\CM}{}$ for an arbitrary $\phi'\in\pFormulas$.
						For the sake of the induction, suppose 
							that 
								if $s\sqsubseteq s'$ and $(\aModalFrame,\mathcal{V}),s\models\phi'$
								then $(\aModalFrame,\mathcal{V}),s'\models\phi'$.
						Further suppose that 
							$s\sqsubseteq s'$ and 
							$(\aModalFrame,\mathcal{V}),s\models\Iproves{M}{\phi'}{\CM}{}$.
						Thus on the one hand, 
							$s\access{\CM}{\CM}{}s'$, by MIAR-inclusion, and
						also, on the other hand, 
							for all $s''\in\states$,
								if $s\access{M}{\CM}{}s''$ 
								then $(\aModalFrame,\mathcal{V}),s''\models\phi'$,
									by definition.
						Now, 
							let $s''\in\states$ and
							further suppose that $s'\access{M}{\CM}{}s''$.
						Hence $s\access{M}{\CM}{}s''$, by special transitivity.
						Hence $(\aModalFrame,\mathcal{V}),s''\models\phi'$.
						(Observe that 
							the induction hypothesis turns out to be irrelevant for this case.)				
			\end{itemize}
	\end{itemize}
\end{proof}
\noindent
Observe that 
	the induction hypothesis in the proof turns out to be irrelevant for 
		all negative connectives and the modality, which
			are made to conserve the monotonicity of atomic propositions 
				(\cf Page~\pageref{page:Monotonicity}).

\begin{proposition}[Concrete LIiP-models]\label{proposition:InputHistoryModel}\ 
	Let $\mathfrak{M}$ designate an arbitrary concrete LIiP-model, \ie a model with
		ingredients like in Definition~\ref{definition:SemanticIngredients}, and 
	let $\phi\in\pFormulas$.
	
	Then, 
		$$\text{$\mathfrak{M},\mathtt{0}\models\phi$ if and only if $\mathfrak{M}\models\phi$\,.}$$
\end{proposition}
\begin{proof}
	The only-if direction follows from the definition of global satisfaction (\cf Definition~\ref{definition:TruthValidity}). 
	For the if-direction,
		consider that the concrete state space $\states=\setst{s\in\states}{\mathtt{0}\sqsubseteq s}$ and
		apply the antecedent Monotonicity Lemma.
\end{proof}

\begin{proposition}[Admissibility of LIiP-specific axioms and rules]\label{proposition:AxiomAndRuleAdmissibility}\ 
	\begin{enumerate}
		\item $\models\knows{a}{a}$
		\item $\models(\knows{a}{M}\land\knows{a}{M'})\lequiv\knows{a}{\pair{M}{M'}}$
		\item $\models\Iproves{M}{\knows{\CM}{M}}{\CM}{}$
		\item $\models(\Iproves{M}{(\phi\limp\phi')}{\CM}{})\limp
				((\Iproves{M}{\phi}{\CM}{})\limp\Iproves{M}{\phi'}{\CM}{})$
		\item $\models(\Iproves{M}{\phi}{\CM}{})\limp(\knows{\CM}{M}\limp\phi)$
		\item $\models(\Iproves{M}{\phi}{\CM}{})\limp(\Iproofdiamond{M}{\phi}{\CM}{})$
		\item $\models\phi\limp\Iproves{M}{\phi}{\CM}{}$
		\item If $\models\phi$ then $\models\Iproves{M}{\phi}{\CM}{}$
		\item If $\models\knows{\CM}{M}\limp\knows{\CM}{M'}$ then 
					$\models(\Iproves{M'}{\phi}{\CM}{})\limp\Iproves{M}{\phi}{\CM}{}$.
	\end{enumerate}
\end{proposition}
\begin{proof}
	1 and 2 are immediate;
	4 and 8 hold by the fact that LIiP has a standard Kripke-semantics; 
	3 follows directly from the epistemic-image property of $\access{M}{\CM}{}$,
	5 from the conditional reflexivity of $\access{M}{\CM}{}$, and 
	9 from the proof-monotonicity property of $\access{M}{\CM}{}$.
	6 follows from 
		the seriality of $\access{M}{\CM}{}$ and  
		the MIAR-inclusion and 
		the epistemic-image property of $\access{M}{\CM}{}$ as follows:
		Let $(\aModalFrame,\mathcal{V})$ designate an arbitrary LIiP-model and let $s\in\states$.
		Further let $s'\in\states$ and suppose that $s\sqsubseteq s'$.
		Now suppose that $(\aModalFrame,\mathcal{V}),s'\models\Iproves{M}{\phi}{\CM}{}$.
		Further let $s''\in\states$ and suppose that $s'\sqsubseteq s''$.
		Hence $(\aModalFrame,\mathcal{V}),s''\models\Iproves{M}{\phi}{\CM}{}$ by the Monotonicity Lemma.
		That is,
			for all $s'''\in\states$,
				if $s''\access{M}{\CM}{}s'''$ then $(\aModalFrame,\mathcal{V}),s'''\models\phi$.	
		But by the seriality of $\access{M}{\CM}{}$, 
			there is indeed an $s'''\in\states$ such that $s''\access{M}{\CM}{}s'''$.
		Hence, 
			$(\aModalFrame,\mathcal{V}),s'''\models\phi$, and also   
			$s''\sqsubseteq s'''$ by the MIAR-inclusion property, and yet also 
			$M\in\clo{\CM}{s'''}(\emptyset)$ by the epistemic-image property.
		Thus $(\aModalFrame,\mathcal{V}),s'''\models\knows{\CM}{M}\land\phi$.
		Hence $(\aModalFrame,\mathcal{V}),s'\models\neg\neg(\knows{\CM}{M}\land\phi)$ and 
		thus $(\aModalFrame,\mathcal{V}),s\models(\Iproves{M}{\phi}{\CM}{})\limp\neg\neg(\knows{\CM}{M}\land\phi)$.	
	Finally, 
		7 follows from 
			the MIAR-inclusion property of $\access{M}{\CM}{}$ and 
			the Monotonicity Lemma (which in turn holds thanks to 
				the special transitivity and the MIAR-inclusion property of $\access{M}{\CM}{}$) as follows: 
		Let $(\aModalFrame,\mathcal{V})$ designate an arbitrary LIiP-model and let $s\in\states$.
		Further let $s'\in\states$ and suppose that $s\sqsubseteq s'$.
		Now suppose that $(\aModalFrame,\mathcal{V}),s'\models\phi$.
		Additionally, let $s''\in\states$ and suppose that $s'\access{M}{\CM}{}s''$.
		Hence $s'\sqsubseteq s''$ by MIAR-inclusion.
		Hence $(\aModalFrame,\mathcal{V}),s''\models\phi$ by the Monotonicity Lemma.
		Thus,  
			$(\aModalFrame,\mathcal{V}),s'\models\Iproves{M}{\phi}{\CM}{}$, and then
			$(\aModalFrame,\mathcal{V}),s\models\phi\limp\Iproves{M}{\phi}{\CM}{}$.
\end{proof}

\begin{theorem}[Axiomatic adequacy]\label{theorem:Adequacy}\ 
	$\LIiPded$ is \emph{adequate} for $\models$, \ie:
	\begin{enumerate}
		\item if $\LIiPded\phi$ then $\models\phi$\quad(axiomatic soundness)
		\item if $\models\phi$ then $\LIiPded\phi$\quad(semantic completeness).
	\end{enumerate}
\end{theorem}
\begin{proof}
	Both parts can be proved with standard means:
	axiomatic soundness follows from 
		the admissibility of the axioms and rules 
			(\cf Propostion~\ref{proposition:AxiomAndRuleAdmissibility}) as usual, and 
	semantic completeness follows by means of a construction of canonical models that 
		is appropriate for intuitionistic normal modal logic as follows.
	
	Let
		\begin{itemize}
			\item $\mathcal{W}$ designate the set of all prime LIiP-consistent sets\footnote{*
				A set $W$ of LIiP-formulas is prime LIiP-consistent :iff 
					$W$ is LIiP-consistent and 
					$W$ is prime.
				A set $W$ of LIiP-formulas is LIiP-consistent :iff 
					$W$ is not LIiP-inconsistent.
				A set $W$ of LIiP-formulas is LIiP-inconsistent :iff 
					there is a finite $W'\subseteq W$ such that $((\bigwedge W')\limp\false)\in\LIiP$.
				A set $W$ of LIiP-formulas is prime :iff 
						first, $W$ is deductively closed, that is, 
							there is a finite $W'\subseteq W$ such that 
								for all $\phi\in\pFormulas$, 
									if $((\bigwedge W')\limp\phi)\in\LIiP$ then $\phi\in W$, and
						second, $W$ has the disjunction property, that is,
							for all $\phi,\phi'\in\pFormulas$,
								if $\phi\lor\phi'\in W$
								then $\phi\in W$ or $\phi'\in W$.
				Similar to a classical Lindenbaum construction 
					(extending consistent sets to \emph{maximal} consistent sets)  
						any LIiP-consistent set can be extended to 
							a \emph{prime} LIiP-consistent set.}
			\item $w\sqsubseteq w'$ :iff $w\subseteq w'$
			\item for all $w,w'\in\mathcal{W}$,
				$w\canrel{M}{\CM}{}w'$ :iff $\setst{\phi\in\pFormulas}{\Iproves{M}{\phi}{\CM}{}\in w}\subseteq w'$
			\item for all $w\in\mathcal{W}$, $w\in\canVal(P)$ :iff $P\in w$.
		\end{itemize}
	Then 
		$\canModel\defeq
			(\mathcal{W},\sqsubseteq,\set{\canrel{M}{\CM}{}}_{M\in\messages},\canVal)$
		designates the \emph{canonical model} for LIiP.
	
	Following standard practice common to all intuitionistic normal modal logics, 
		the following useful property of $\canModel$, 
			the so-called \emph{Truth Lemma,} 
			$$\text{for all $\phi\in\pFormulas$ and $w\in\mathcal{W}$,
				$\phi\in w$ if and only if $\canModel,w\models\phi$}$$
		can be proved by induction on the structure of $\phi$.
	With this lemma, 
		it can then be proved that 
			for all $\phi\in\pFormulas$,
				if $\not\LIiPded\phi$ then $\not\models\phi$.
	Let 
		$\phi\in\pFormulas$, and 
	suppose that 
		$\not\LIiPded\phi$.
	Thus, 
		$\set{\neg\phi}$ 
			is LIiP-consistent, and 
			can be extended to a prime LIiP-consistent set $w$, \ie
				$\neg\phi\in w\in\mathcal{W}$.
	Hence 
		$\canModel,w\models\neg\phi$, by the Truth Lemma.
	Thus: 
		$\canModel,w\not\models\phi$,
		$\canModel\not\models\phi$, and
		$\not\models\phi$.
	That is,
			$\canModel$ is a 
				\emph{universal} (for \emph{all} $\phi\in\pFormulas$) 
				\emph{counter-model} (if $\phi$ is a non-theorem then $\canModel$ falsifies $\phi$).
	
	The only proof obligation specific to the semantic-completeness proof for LIiP is to prove that 
		$\canModel$ is also an \emph{LIiP-model}. 
	So let us instantiate our data-mining operator $\clo{a}{}$ (\cf Page~\pageref{page:DataMining}) on $\mathcal{W}$  
		by letting for all $w\in\mathcal{W}$
			$$\msgs{a}(w)\defeq\setst{M}{\knows{a}{M}\in w},$$ and let us prove that:
					\begin{enumerate}
				\item If $w\canrel{M}{\CM}{}w'$ then $M\in\clo{\CM}{w'}(\emptyset)$
				\item If $M\in\clo{\CM}{w}(\emptyset)$ then 
						$w\canrel{M}{\CM}{}w$
				\item there is $w'\in\mathcal{W}$ such that $w\canrel{M}{\CM}{}w'$
				\item ${\canrel{M}{\CM}{}}\subseteq{\canrel{\CM}{\CM}{}}={\sqsubseteq}$
				\item $({\canrel{\CM}{\CM}{}}\circ{\canrel{M}{\CM}{}})\subseteq{\canrel{M}{\CM}{}}$
				\item If $M\sqsubseteq_{\CM}M'$ 
						then ${\canrel{M}{\CM}{}}\subseteq{\canrel{M'}{\CM}{}}$
					\end{enumerate} 
	For (1),
		let $w,w'\in\mathcal{W}$ and 
		suppose that $w\canrel{M}{\CM}{}w'$.
	That is,	
		for all $\phi\in\pFormulas$, 
			if $\Iproves{M}{\phi}{\CM}{}\in w$
			then $\phi\in w'$.
	Since $w$ is deductively closed,
		$$\text{$\Iproves{M}{\knows{\CM}{M}}{\CM}{}\in w$\quad(self-knowledge).}$$
	Hence 
		$\knows{\CM}{M}\in w'$.
	Thus 
		$M\in\clo{\CM}{w'}(\emptyset)$ by the definition of $\clo{\CM}{w'}$.

	For (2), 
		let $w\in\mathcal{W}$ and 
		suppose that 
			$M\in\clo{\CM}{w}(\emptyset)$.
	Hence $\knows{\CM}{M}\in w$ due to the deductive closure of $w$.
	Further suppose that $\Iproves{M}{\phi}{\CM}{}\in w$.
	Since $w$ is deductively closed,
		$$\text{$(\Iproves{M}{\phi}{\CM}{})\limp(\knows{\CM}{M}\limp\phi)\in w$\quad(ET).}$$
	Hence, 
		$\knows{\CM}{M}\limp\phi\in w$, and
		$\phi\in w$, by consecutive  \emph{modus ponens.}

	For (3),
		let $w\in\mathcal{W}$ and $\phi\in\pFormulas$, and
		suppose that $\Iproves{M}{\phi}{\CM}{}\in w$.
	Since $w$ is deductively closed,
		$$\text{$(\Iproves{M}{\phi}{\CM}{})\limp\Iproofdiamond{M}{\phi}{\CM}{}\in w$\quad(ID).}$$
	Hence, $\Iproofdiamond{M}{\phi}{\CM}{}\in w$ by \emph{modus ponens.}
	That is, 
		$\neg\neg(\knows{\CM}{M}\land\phi)\in w$ by definition.
	Since $w$ is deductively closed,
		$\neg\neg(\knows{\CM}{M}\land\phi)\limp\neg\neg\phi\in w$.
	Hence 
		$\neg\neg\phi\in w$ by \emph{modus ponens.}
	Hence $\canModel, w\models\neg\neg\phi$ by the Truth Lemma.
	Hence for all $w'\in\mathcal{W}$,
		if $w\sqsubseteq w'$ then 
		there is $w''\in\mathcal{W}$ such that 
			$w'\sqsubseteq w''$ and 
			$\canModel, w''\models\phi$ by definition.	
	Hence $\phi\in w''$ by the Truth Lemma.
		
	For (4), 
		let us first prove that ${\canrel{\CM}{\CM}{}}={\sqsubseteq}$\,.
	So let $w,w'\in\mathcal{W}$ and 
	suppose that $w\canrel{\CM}{\CM}{}w'$.
	That is,	
		for all $\phi\in\pFormulas$, 
			if $\Iproves{\CM}{\phi}{\CM}{}\in w$
			then $\phi\in w'$.
	Further 
		let $\phi\in\pFormulas$ and 
		suppose that $\phi\in w$.
	Hence if $\Iproves{\CM}{\phi}{\CM}{}\in w$ then $\phi\in w'$.
	Since $w$ is deductively closed,
		$$(\Iproves{\CM}{\phi}{\CM}{})\lequiv\phi\in w\quad\text{(TMM)}$$
	Hence $\Iproves{\CM}{\phi}{\CM}{}\in w$ by LIiP-\emph{modus ponens}.
	Hence $\phi\in w'$ by meta-level \emph{modus ponens}.
	Thus $w\subseteq w'$, and then $w\sqsubseteq w'$.
	Conversely, suppose that $w\sqsubseteq w'$.
	That is, $w\subseteq w'$.
	Further 
		let $\phi\in\pFormulas$ and 
		suppose that $\Iproves{\CM}{\phi}{\CM}{}\in w$.
	Hence $\Iproves{\CM}{\phi}{\CM}{}\in w'$.
	Since $w'$ is deductively closed,
		$$(\Iproves{\CM}{\phi}{\CM}{})\lequiv\phi\in w'\quad\text{(TMM)}$$
	Hence $\phi\in w'$ by LIiP-\emph{modus ponens}.
	Second, let us prove that ${\canrel{M}{\CM}{}}\subseteq{\canrel{\CM}{\CM}{}}$\,.
	So let $w,w'\in\mathcal{W}$ and 
	suppose that $w\canrel{M}{\CM}{}w'$.
	That is,	
		for all $\phi\in\pFormulas$, 
			if $\Iproves{M}{\phi}{\CM}{}\in w$
			then $\phi\in w'$.
	Further 
		let $\phi\in\pFormulas$ and 
		suppose that $\Iproves{\CM}{\phi}{\CM}{}\in w$.
	Since $w$ is deductively closed,
		$$(\Iproves{\CM}{\phi}{\CM}{})\lequiv\phi\in w\quad\text{(TMM)}$$
	Hence $\phi\in w$ by LIiP-\emph{modus ponens}.
	Since $w$ is deductively closed,
		$$\phi\limp\Iproves{M}{\phi}{\CM}{}\in w\quad\text{(MM)}$$
	Hence $\Iproves{M}{\phi}{\CM}{}\in w$ by LIiP-\emph{modus ponens}.
	Hence $\phi\in w'$ by meta-level \emph{modus ponens}.
	
	For (5),
		let $w,w'\in\mathcal{W}$ and 
		suppose that $w\mathrel{(\canrel{\CM}{\CM}{}\circ\canrel{M}{\CM}{})}w'$.
	That is,	
		there is $w''\in\mathcal{W}$ such that
			$w\canrel{\CM}{\CM}{}w''$ and 
			$w''\canrel{M}{\CM}{}w'$.
	Thus, 
		(for all $\phi\in\pFormulas$, 
			if $\Iproves{\CM}{\phi}{\CM}{}\in w$
			then $\phi\in w''$) and 
		(for all $\phi\in\pFormulas$, 
			if $\Iproves{M}{\phi}{\CM}{}\in w''$
			then $\phi\in w'$).
	Further let $\phi\in\pFormulas$ and suppose that $\Iproves{M}{\phi}{\CM}{}\in w$.
	Since $w$ is deductively closed,
		$$(\Iproves{\CM}{(\Iproves{M}{\phi}{\CM}{})}{\CM}{})\lequiv
			\Iproves{M}{\phi}{\CM}{}\in w\quad\text{(TMM)}$$
	Hence $(\Iproves{\CM}{(\Iproves{M}{\phi}{\CM}{})}{\CM}{})\in w$.
	Hence $\Iproves{M}{\phi}{\CM}{}\in w''$ by the first hypothesis.
	Hence $\phi\in w'$ by the second hypothesis.
	Thus $w\canrel{M}{\CM}{}w'$.

	For (6),
		suppose that 
			$M\sqsubseteq_{\CM}M'$.
	That is,
		for all $w\in\mathcal{W}$, 
			if $M\in\clo{\CM}{w}(\emptyset)$ then $M'\in\clo{\CM}{w}(\emptyset)$.
	Hence for all $w\in\mathcal{W}$,
		if $\knows{\CM}{M}\in w$ then $\knows{\CM}{M'}\in w$
				due to the deductive closure of $w$, which 
					contains all the term axioms corresponding to the defining clauses of $\clo{\CM}{w}$.
	Hence for all $w\in\mathcal{W}$,
			if $\canModel, w\models\knows{\CM}{M}$ then $\canModel, w\models\knows{\CM}{M'}$, by the Truth Lemma.
	Hence also, 
		for all $w\in\mathcal{W}$, $\canModel, w\models\knows{\CM}{M}\limp\knows{\CM}{M'}$ by
			the definition of $\sqsubseteq$.
	Hence for all $w\in\mathcal{W}$, 
		$\knows{\CM}{M}\limp\knows{\CM}{M'}\in w$ by the Truth Lemma.
	Hence the following intermediate result, called IR, 
			$$\text{for all $w\in\mathcal{W}$ and $\phi\in\pFormulas$, 
			$(\Iproves{M'}{\phi}{\CM}{})\limp\Iproves{M}{\phi}{\CM}{}\in w$,}$$ by EA.
	Now, 
		let $w,w'\in\mathcal{W}$ and 
		suppose that $w\canrel{M}{\CM}{}w'$.
	That is,
		for all $\phi\in\pFormulas$,
			if $\Iproves{M}{\phi}{\CM}{}\in w$ 
			then $\phi\in w'$.
	Hence,  
		for all $\phi\in\pFormulas$,
					if $\Iproves{M'}{\phi}{\CM}{}\in w$
					then $\phi\in w'$ by IR. 
	Hence $w\canrel{M'}{\CM}{}w'$ by definition.
\end{proof}
			
\begin{theorem}[Finite-model property]\label{theorem:FiniteModelProperty}
	For any LIiP-model $\mathfrak{M}$,
	if $\mathfrak{M}, s\models\phi$ 
	then there is a finite LIiP-model $\mathfrak{M}_{\mathrm{fin}}$ such that 
		$\mathfrak{M}_{\mathrm{fin}}, s\models\phi$.
\end{theorem}
\begin{proof}
	By the fact that 
		the \emph{minimal filtration} \cite{ModalModelTheory}
			$$\mathfrak{M}_{\mathrm{flt}}^{\mathrm{min},\Gamma}\defeq
				(\states/_{\sim_{\Gamma}},
					\sqsubseteq^{\mathrm{min},\Gamma},
					\set{\access{M}{\CM}{\mathrm{min},\Gamma}}_{M\in\messages},\mathcal{V}_{\Gamma})$$ of 
			any LIiP-model $\mathfrak{M}\defeq(\states,\sqsubseteq,\set{\access{M}{\CM}{}}_{M\in\messages},\mathcal{V})$ 
			through a finite $\Gamma\subseteq\pFormulas$ is a finite LIiP-model such that 
				for all $\gamma\in\Gamma$, 
					$\mathfrak{M},s\models\gamma$ if and only if 
					$\mathfrak{M}_{\mathrm{flt}}^{\mathrm{min},\Gamma},[s]_{\sim_{\Gamma}}\models\gamma$.
			Following \cite{ModalModelTheory} for our setting, 
				we define 
				\begin{eqnarray*}
					{\sim_{\Gamma}}&\defeq&
						\setst{(s,s')\in\states\times\states}{\text{for all $\gamma\in\Gamma$, 
							$\mathfrak{M},s\models\gamma$ iff $\mathfrak{M},s'\models\gamma$}}\\ 
					{\sqsubseteq^{\mathrm{min},\Gamma}} &\defeq&
						\setst{([s]_{\sim_{\Gamma}},[s']_{\sim_{\Gamma}})}{(s,s')\in{\sqsubseteq}}\\
					{\access{M}{\CM}{\mathrm{min},\Gamma}} &\defeq&
						\setst{([s]_{\sim_{\Gamma}},[s']_{\sim_{\Gamma}})}{(s,s')\in{\access{M}{\CM}{}}}\\
					\mathcal{V}_{\Gamma}(P)&\defeq&
							\setst{[s]_{\sim_{\Gamma}}}{s\in\mathcal{V}(P)}\,.
				\end{eqnarray*}
			We further fix 
				$M\in\clo{\CM}{[s]_{\sim_{\Gamma}}}(\emptyset)$ :iff 
				$[s]_{\sim_{\Gamma}}\in\mathcal{V}_{\Gamma}(\knows{\CM}{M})$, and
				choose $\Gamma$ to be the (finite) sub-formula closure of $\phi$.
			Hence, we are left to prove that 
				$\mathfrak{M}_{\mathrm{flt}}^{\mathrm{min},\Gamma}$ is indeed an LIiP-model, which
				means that we are left to prove that 
						$\sqsubseteq^{\mathrm{min},\Gamma}$ and ${\access{M}{a}{\mathrm{min},\Gamma}}$ have 
					all the properties stipulated by the semantic interface of LIiP;
					this is straightforward and therefore relegated to 
						Appendix~\ref{appendix:FiniteModelProperty}.
\end{proof}

\begin{corollary}[Algorithmic decidability]\label{corollary:AlgorithmicDecidability}
	LIiP is algorithmically decidable.
\end{corollary}
\begin{proof}
	In order to algorithmically decide whether or not $\phi\in\LIiP$ 
	(that is, $\LIiPded\phi$) for some $\phi\in\pFormulas$ (and the current choice of $\messages$), 
		axiomatic adequacy allows us to check whether or not $\neg\phi$ is locally satisfiable (That is,
			whether or not $\mathfrak{M},s\models\neg\phi$ for 
				some LIiP-model $\mathfrak{M}$ and state $s$.
					Also, $M\in\clo{\CM}{s}(\emptyset)$ on the currently chosen  
						message language $\messages$ is obviously decidable;
						for other, more complex message languages including cryptographic messages, 
							see for example \cite{TGJ} and \cite{BaskarJamSuresh2010}).
	But then, the finite-model property of LIiP allows us 
		to enumerate all finite LIiP-models $\mathfrak{M}_{\mathrm{fin}}$ up to a size of at most 2 to the power 
			of the size $n$ of the sub-formula closure of $\neg\phi$ and
		to check whether or not $\mathfrak{M}_{\mathrm{fin}},s\models\neg\phi$.
	(First, there are at most $2^n$ equivalence classes for $n$ formulas.
	 Second, 
		checking intuitionistic negation, which 
			is checking classical negation within the up-set of the state $s$ with respect to $\sqsubseteq$, 
				within a finite model is also a finite task.)
\end{proof}
Note that 
	the algorithmic complexity of LIiP will depend on 
		the specific choice of $\messages$ and the correspondingly chosen term axioms.

\section{Conclusion}
We have produced LIiP from LiP with 
	as main contributions those described in Section~\ref{section:Contribution}. 
In future work,
	we shall work out 
	dynamic and first-order extensions of LIiP.
As roughly related work,
	we
		have already mentioned \cite{LIPArtemov} 
			(\cf Page~\pageref{page:LIP:1}, 
						\pageref{page:LIP:2}, 
						\pageref{page:LIP:4}, and 
						\pageref{page:LIP:7}) and 
		can further mention 
			\cite{PouliasisPrimieroIMLA2013} and 
			\cite{SterenBonelliIMLA2013}.
In \cite{PouliasisPrimieroIMLA2013},
	a fragment of an intuitionistic version of the minimal Justification Logic \cite{JustificationLogic} (mJL) is introduced in 
		the context of an ambitious Curry-Howard isomorphism for modular programming.
A similar appreciation can be made of \cite{SterenBonelliIMLA2013}, which
	introduced an intuitionistic fragment of 
		the Logic of (non-interactive) Proofs (extending mJL) \cite{LIPArtemov}.
The main contribution of 
	\cite{PouliasisPrimieroIMLA2013} as well as \cite{SterenBonelliIMLA2013} seems to be  
		a programming calculus for an axiomatically defined intuitionistic modal logic rather than 
			logic itself, whereas 
		ours is an intuitionistic modal logic  
			with an axiomatics and a set-theoretically constructive semantics.

\paragraph{Acknowledgements}
I thank Olga Grinchtein for spotting a few typos.

\bibliographystyle{alpha}

\begin{thebibliography}{FHMV95}

\bibitem[AI07]{LIPArtemov}
S.~Artemov and R.~Iemhoff.
\newblock The basic intuitionistic logic of proofs.
\newblock {\em The Journal of Symbolic Logic}, 72(2), 2007.

\bibitem[And08]{SecurityEngineering}
R.~Anderson.
\newblock {\em Security Engineering: A Guide to Building Dependable Distributed
  Systems}.
\newblock Wiley, second edition, 2008.

\bibitem[Art08]{JustificationLogic}
S.~Artemov.
\newblock The logic of justifications.
\newblock {\em The Review of Symbolic Logic}, 1(4), 2008.

\bibitem[AtC07]{HybridLogics}
C.~Areces and B.~ten Cate.
\newblock {\em Handbook of Modal Logic}, chapter Hybrid Logics.
\newblock Volume~3 of Blackburn et~al. \cite{HBModalLogic}, 2007.

\bibitem[BG07]{FOModalLogic}
T.~Bra\"{u}ner and S.~Ghilardi.
\newblock {\em Handbook of Modal Logic}, chapter First-Order Modal Logic.
\newblock Volume~3 of Blackburn et~al. \cite{HBModalLogic}, 2007.

\bibitem[BRS10]{BaskarJamSuresh2010}
A.~Baskar, R.~Ramanujam, and S.P. Suresh.
\newblock A {DEXPTIME}-complete {Dolev-Yao} theory with distributive
  encryption.
\newblock In {\em Proceedings of {MFCS}}, volume 6281 of {\em LNCS}. Springer,
  2010.

\bibitem[BvB07]{ModalLogicSemanticPerspective}
P.~Blackburn and J.~van Benthem.
\newblock {\em Handbook of Modal Logic}, chapter Modal Logic: A Semantic
  Perspective.
\newblock Volume~3 of Blackburn et~al. \cite{HBModalLogic}, 2007.

\bibitem[BvBW07]{HBModalLogic}
P.~Blackburn, J.~van Benthem, and F.~Wolter, editors.
\newblock {\em Handbook of Modal Logic}, volume~3 of {\em Studies in Logic and
  Practical Reasoning}.
\newblock Elsevier, 2007.

\bibitem[Do{\v s}84]{IntuitionisticDoubleNegationAsNecessity}
K.~Do{\v s}en.
\newblock Intuitionistic double negation as a necessity operator.
\newblock {\em Publications de l'Institut Math{\'e}matique (Beograd)}, 35(49),
  1984.

\bibitem[dPR11]{BasicConstructiveModality}
V.~de~Paiva and E.~Ritter.
\newblock {\em Logic without Frontiers: Festschrift for Walter Alexandre
  Carnielli on the occasion of his 60th birthday}, volume~17 of {\em Tributes},
  chapter Basic Constructive Modality.
\newblock College Publications, 2011.

\bibitem[Fef89]{TheNumberSystems}
S.~Feferman.
\newblock {\em The Number Systems: Foundations of Algebra and Analysis}.
\newblock AMS Chelsea Publishing, second edition, 1964 (1989).
\newblock Reprinted by the American Mathematical Society, 2003.

\bibitem[FHMV95]{Epistemic_Logic}
R.~Fagin, J.Y. Halpern, Y.~Moses, and M.Y. Vardi.
\newblock {\em Reasoning about Knowledge}.
\newblock MIT Press, 1995.

\bibitem[{Fis}84]{IK:FischerServi}
G.~{Fischer Servi}.
\newblock Axiomatisations for some intuitionistic modal logics.
\newblock {\em Rendiconti del seminario matematico del Politecnico di Torino},
  42(3), 1984.

\bibitem[Fit07]{ModalProofTheory}
M.~Fitting.
\newblock {\em Handbook of Modal Logic}, chapter Modal Proof Theory.
\newblock Volume~3 of Blackburn et~al. \cite{HBModalLogic}, 2007.

\bibitem[FSK10]{CryptographyEngineering}
N.~Ferguson, B.~Schneier, and T.~Kohno.
\newblock {\em Cryptography Engineering: Design Principles and Practical
  Applications}.
\newblock Wiley, 2010.

\bibitem[Gab95]{WhatIsALogicalSystem}
D.M. Gabbay, editor.
\newblock {\em What Is a Logical System?}
\newblock Number~4 in Studies in Logic and Computation. Oxford University
  Press, 1995.

\bibitem[GO07]{ModalModelTheory}
V.~Goranko and M.~Otto.
\newblock {\em Handbook of Modal Logic}, chapter Model Theory of Modal Logic.
\newblock Volume~3 of Blackburn et~al. \cite{HBModalLogic}, 2007.

\bibitem[Gol11]{ComputerSecurity}
D.~Gollmann.
\newblock {\em Computer Security}.
\newblock Wiley, third edition, 2011.

\bibitem[HR10]{EpistemicLogicFiveQuestions}
V.F. Hendricks and O.~Roy, editors.
\newblock {\em Epistemic Logic: 5 Questions}.
\newblock Automatic Press, 2010.

\bibitem[Hru07]{LowerProofComplexityIL}
P.~Hrube{\v s}.
\newblock A lower bound for intuitionistic logic.
\newblock {\em Annals of Pure and Applied Logic}, 146, 2007.

\bibitem[Je{\v r}08]{jerabek:independent-bases}
E.~Je{\v r}{\'a}bek.
\newblock Independent bases of admissible rules.
\newblock {\em Logic Journal of the Interest Group in Pure and Applied Logic},
  16(3), 2008.

\bibitem[Kra12a]{LiP}
S.~Kramer.
\newblock A logic of interactive proofs (formal theory of knowledge transfer).
\newblock Technical Report 1201.3667, arXiv, 2012.
\newblock \url{http://arxiv.org/abs/1201.3667}.

\bibitem[Kra12b]{LDiiP}
S.~Kramer.
\newblock Logic of negation-complete interactive proofs (formal theory of
  epistemic deciders).
\newblock Technical Report 1208.5913, arXiv, 2012.
\newblock \url{http://arxiv.org/abs/1208.5913}.

\bibitem[Kra12c]{LiiP}
S.~Kramer.
\newblock Logic of non-monotonic interactive proofs (formal theory of temporary
  knowledge transfer).
\newblock Technical Report 1208.1842, arXiv, 2012.
\newblock \url{http://arxiv.org/abs/1208.1842}.

\bibitem[Kra13a]{ILIiP}
S.~Kramer.
\newblock Logic of intuitionistic interactive proofs (formal theory of
  disjunctive knowledge transfer).
\newblock Short paper presented at the Congress on Logic and Philosophy of
  Science, Ghent, 2013.

\bibitem[Kra13b]{KramerIMLA2013}
S.~Kramer.
\newblock Logic of negation-complete interactive proofs (formal theory of
  epistemic deciders).
\newblock In {\em Proceedings of {IMLA}}, volume 300 of {\em ENTCS}. Elsevier,
  2013.

\bibitem[Kra13c]{KramerICLA2013}
S.~Kramer.
\newblock Logic of non-monotonic interactive proofs.
\newblock In {\em Proceedings of {ICLA}}, volume 7750 of {\em LNCS}. Springer,
  2013.

\bibitem[Kri65]{KripkeSemanticsIL}
S.A. Kripke.
\newblock {\em Formal Systems and Recursive Functions}, volume~40 of {\em
  Studies in Logic and the Foundations of Mathematics}, chapter Semantical
  Analysis of Intuitionistic Logic I.
\newblock Elsevier, 1965.

\bibitem[Mos06]{NotesOnSetTheory}
Y.~Moschovakis.
\newblock {\em Notes on Set Theory}.
\newblock Springer, 2nd edition, 2006.

\bibitem[Mos10]{sep-logic-intuitionistic}
J.~Moschovakis.
\newblock Intuitionistic logic.
\newblock In {\em The Stanford Encyclopedia of Philosophy}. Summer 2010
  edition, 2010.

\bibitem[MV07]{MultiAgents}
J.-J. Meyer and F.~Veltman.
\newblock {\em Handbook of Modal Logic}, chapter Intelligent Agents and Common
  Sense Reasoning.
\newblock Volume~3 of Blackburn et~al. \cite{HBModalLogic}, 2007.

\bibitem[PP13]{PouliasisPrimieroIMLA2013}
K.~Pouliasis and G.~Primiero.
\newblock {J-Calc}: A typed lambda calculus for {I}ntuitionistic
  {J}ustification {L}ogic.
\newblock In {\em Proceedings of {IMLA}}, volume 300 of {\em ENTCS}. Elsevier,
  2013.

\bibitem[PS86]{IK:PlotkinStirling}
G.~Plotkin and C.~Stirling.
\newblock A framework for intuitionistic modal logics.
\newblock In {\em Proceedings of the Conference on Theoretical Aspects of
  Rationality and Knowledge}. Morgan Kaufmann Publishers Inc., 1986.

\bibitem[Ran10]{CKembeddedIntoIK}
K.~Ranalter.
\newblock Embedding constructive {K} into intuitionistic {K}.
\newblock {\em Electronic Notes in Theoretical Computer Science}, 262, 2010.

\bibitem[SB13]{SterenBonelliIMLA2013}
G.~Steren and E.~Bonelli.
\newblock {I}ntuitionistic {H}ypothetical {L}ogic of {P}roofs.
\newblock In {\em Proceedings of {IMLA}}, volume 300 of {\em ENTCS}. Elsevier,
  2013.

\bibitem[Sim94]{PhDThesisSimpson}
A.K. Simpson.
\newblock {\em The Proof-Theory and Semantics of Intuitionistic Modal Logic}.
\newblock PhD thesis, University of Edinburgh, 1994.

\bibitem[Sta79]{ILcomplexity}
R.~Statman.
\newblock Intuitionistic propositional logic is polynomial-space complete.
\newblock {\em Theoretical Computes Science}, 9, 1979.

\bibitem[Tay99]{PracticalFoundationsOfMathematics}
P.~Taylor.
\newblock {\em Practical Foundations of Mathematics}.
\newblock Cambridge University Press, 1999.

\bibitem[TGD10]{TGJ}
A.~Tiu, R.~Gor\'{e}, and J.~Dawson.
\newblock A proof theoretic analysis of intruder theories.
\newblock {\em Logical Methods in Computer Science}, 6(3), 2010.

\bibitem[vB97]{ModalLogicInformation}
J.~van Benthem.
\newblock {\em Logic and Reality: Essays on the Legacy of Arthur Prior},
  chapter Modal Logic as a Theory of Information.
\newblock Clarendon Press, Oxford, 1997.

\bibitem[vB09]{InformationIL}
J.~van Benthem.
\newblock The information in intuitionistic logic.
\newblock {\em Synthese}, 167, 2009.

\bibitem[Wij90]{ConstructiveModalLogicI}
D.~Wijesekera.
\newblock Constructive modal logic {I}.
\newblock {\em Annals of Pure and Applied Logic}, 50, 1990.

\end{thebibliography}

\appendix

\section{Remaining proofs}\label{appendix:Proofs}
\subsection{Proof of Theorem~\ref{theorem:SomeUsefulDeducibleLogicalLaws}}\label{appendix:Proofs:LogicalLaws}
	For 0, combine MM and MP.
	For 1--6, 8--9, and 32--35, 
		consult their analogs and their proofs in LiP \cite{LiP}; they  
			require no non-intuitionistic machinery.
	For 7,
		apply ID to 6.
	For 10, combine ET and MM.
	For 11, 
		consider that: 
			\begin{enumerate}
				\item $\LIiPded\Iproves{M}{\knows{\CM}{M}}{\CM}{}$\hfill self-knowledge
				\item $\LIiPded(\Iproves{M}{\knows{\CM}{M}}{\CM}{})\limp\Iproofdiamond{M}{\knows{\CM}{M}}{\CM}{}$\hfill ID
				\item $\LIiPded(\Iproofdiamond{M}{\knows{\CM}{M}}{\CM}{})\lequiv\neg\neg(\knows{\CM}{M}\land\knows{\CM}{M})$\hfill definition
				\item $\LIiPded\neg\neg(\knows{\CM}{M}\land\knows{\CM}{M})\lequiv\neg\neg(\knows{\CM}{M})$\hfill IL
				\item $\LIiPded\neg\neg(\knows{\CM}{M})$\hfill 1--4, IL.
			\end{enumerate}
	For 12, 
		consider that:
			\begin{enumerate}
				\item $\LIiPded\knows{\CM}{M}\limp((\Iproves{M}{\phi}{\CM}{})\lequiv\phi)$\hfill ET \emph{bis}
				\item $\LIiPded\neg\neg(\knows{\CM}{M})$\hfill CMMC
				\item $\LIiPded\neg\neg((\Iproves{M}{\phi}{\CM}{})\lequiv\phi)$\hfill 1, 2, IL.
			\end{enumerate}
	For 13,
		consider that: 
			\begin{enumerate}
				\item $\LIiPded\neg\phi\limp\neg(\knows{\CM}{M}\land\phi)$\hfill IL
				\item $\LIiPded\neg(\knows{\CM}{M}\land\phi)\lequiv\neg\neg\neg(\knows{\CM}{M}\land\phi)$\hfill IL (triple-negation law)
				\item $\LIiPded\neg\neg\neg(\knows{\CM}{M}\land\phi)\lequiv\neg(\Iproofdiamond{M}{\phi}{\CM}{})$\hfill definition
				\item $\LIiPded\neg\phi\limp\neg(\Iproofdiamond{M}{\phi}{\CM}{})$\hfill 1--3, IL.
			\end{enumerate}
	For 14,
		consider that:
			\begin{enumerate}
				\item $\LIiPded\neg(\Iproofdiamond{M}{\phi}{\CM}{})\lequiv\neg\neg\neg(\knows{\CM}{M}\land\phi)$\hfill definition
				\item $\LIiPded\neg\neg\neg(\knows{\CM}{M}\land\phi)\lequiv\neg(\knows{\CM}{M}\land\phi)$\hfill IL (triple-negation law)
				\item $\LIiPded\neg(\knows{\CM}{M}\land\phi)\lequiv\neg(\knows{\CM}{M}\land\Iproves{M}{\phi}{\CM}{})$\hfill 2, ET \emph{bis}, IL
				\item $\LIiPded\neg(\knows{\CM}{M}\land\Iproves{M}{\phi}{\CM}{})\lequiv
									((\Iproves{M}{\phi}{\CM}{})\limp\neg(\knows{\CM}{M}))$\hfill IL 
				\item $\LIiPded((\Iproves{M}{\phi}{\CM}{})\limp\neg(\knows{\CM}{M}))\lequiv
									((\Iproves{M}{\phi}{\CM}{})\limp\false)$\hfill 4, CMMC, IL
				\item $\LIiPded((\Iproves{M}{\phi}{\CM}{})\limp\false)\lequiv
									\neg(\Iproves{M}{\phi}{\CM}{})$\hfill IL
				\item $\LIiPded\neg(\Iproofdiamond{M}{\phi}{\CM}{})\lequiv
								\neg(\Iproves{M}{\phi}{\CM}{})$\hfill 1--6, IL.
			\end{enumerate}	
	For 15, inspect 13 and 14.
	And 16, 17, and 18, are instances of 13, 14, and 15.
			
	For 19,
		consider ID as well as 16 and 18, and that: 
			\begin{enumerate}
				\item $\LIiPded(\Iproofdiamond{M}{\phi}{\CM}{})\lequiv\neg\neg(\knows{\CM}{M}\land\phi)$\hfill definition
				\item $\LIiPded\neg\neg(\knows{\CM}{M}\land\phi)\limp 
						(\neg\neg(\knows{\CM}{M})\land\neg\neg\phi)$\hfill IL
				\item $\LIiPded(\neg\neg(\knows{\CM}{M})\land\neg\neg\phi)\limp\neg\neg\phi$\hfill IL
				\item $\LIiPded(\Iproofdiamond{M}{\phi}{\CM}{})\limp\neg\neg\phi$\hfill 1--3, IL.
			\end{enumerate}
	For 20,
		consider that:
			\begin{enumerate}
				\item $\LIiPded(\Iproofdiamond{M}{\phi}{\CM}{})\limp\neg\neg\phi$\hfill EWDN
				\item $\LIiPded\phi'\limp\Iproves{M}{\phi'}{\CM}{}$\hfill MM
				\item $\LIiPded(\neg\neg\phi\limp\phi')\limp((\Iproofdiamond{M}{\phi}{\CM}{})\limp\Iproves{M}{\phi'}{\CM}{})$\hfill 1, 2, IL.
			\end{enumerate}
	For 21,
		consider that:
			\begin{enumerate}
				\item $\LIiPded(\neg\neg\phi\limp\phi)\limp((\Iproofdiamond{M}{\phi}{\CM}{})\limp\Iproves{M}{\phi}{\CM}{})$\hfill CF
				\item $\LIiPded(\Iproves{M}{\phi}{\CM}{})\limp\Iproofdiamond{M}{\phi}{\CM}{}$\hfill ID
				\item $\LIiPded(\neg\neg\phi\limp\phi)\limp((\Iproofdiamond{M}{\phi}{\CM}{})\lequiv\Iproves{M}{\phi}{\CM}{})$\hfill 1, 2, IL.
			\end{enumerate}			
	
	For 22, inspect ET \emph{bis;} 
	for 23, 22 and 21;
	for 24 and 25, $\LIiPded\neg\false$ and 15 and 13, respectively; 
	for 26, MM and ID.
	For 27, 
		suppose that $\LIiPded\knows{\CM}{M}\limp\phi$.
	Hence,
		$\LIiPded(\Iproves{M}{\knows{\CM}{M}}{\CM}{})\limp\Iproves{M}{\phi}{\CM}{}$ by R, and
		$\LIiPded\Iproves{M}{\phi}{\CM}{}$ by self-knowledge.
	Conversely suppose that
		$\LIiPded\Iproves{M}{\phi}{\CM}{}$.
	Hence $\LIiPded\knows{\CM}{M}\limp\phi$ by ET.
	For 28, instantiate $\phi$ in 27 with $\knows{\CM}{M'}$.
	For 29, inspect 10 and 27, and 
	for 30 and 31, MM.
	For 36, 
		consider that:
			\begin{enumerate}
				\item $\LIiPded(\Iproves{M}{\phi}{\CM}{})\limp\Iproves{M}{(\Iproves{M}{\phi}{\CM}{})}{\CM}{}$\hfill (4)
				\item $\LIiPded\Iproves{M}{((\Iproves{M}{\phi}{\CM}{})\lequiv\phi)}{\CM}{}$\hfill ET \emph{bis} self-proof
				\item $\LIiPded\Iproves{M}{(((\Iproves{M}{\phi}{\CM}{})\limp\phi)\land(\phi\limp\Iproves{M}{\phi}{\CM}{}))}{\CM}{}$\hfill 2, definition
				\item $\LIiPded\Iproves{M}{((\Iproves{M}{\phi}{\CM}{})\limp\phi)}{\CM}{}$\hfill 3, proof conjunctions \emph{bis}
				\item $\LIiPded(\Iproves{M}{((\Iproves{M}{\phi}{\CM}{})\limp\phi)}{\CM}{})\limp
					((\Iproves{M}{(\Iproves{M}{\phi}{\CM}{})}{\CM}{})\limp\Iproves{M}{\phi}{\CM}{})$\hfill K
				\item $\LIiPded(\Iproves{M}{(\Iproves{M}{\phi}{\CM}{})}{\CM}{})\limp\Iproves{M}{\phi}{\CM}{}$\hfill 4, 5, IL
				\item $\LIiPded(\Iproves{M}{(\Iproves{M}{\phi}{\CM}{})}{\CM}{})\lequiv\Iproves{M}{\phi}{\CM}{}$\hfill 1, 6, IL.	
			\end{enumerate}
	For 37, 
		consider that:
			\begin{enumerate}
				\item $\LIiPded(\Iproofdiamond{M}{(\Iproofdiamond{M}{\phi}{\CM}{})}{\CM}{})\lequiv
						\neg\neg(\knows{\CM}{M}\land\neg\neg(\knows{\CM}{M}\land\phi))$\hfill definition
				\item $\LIiPded\begin{array}{@{}l@{}}
						\neg\neg(\knows{\CM}{M}\land\neg\neg(\knows{\CM}{M}\land\phi))\limp\\
						(\neg\neg(\knows{\CM}{M})\land\neg\neg\neg\neg(\knows{\CM}{M}\land\phi))
						\end{array}$\hfill IL
				\item $\LIiPded\begin{array}{@{}l@{}}
						(\neg\neg(\knows{\CM}{M})\land\neg\neg\neg\neg(\knows{\CM}{M}\land\phi))\limp
						\neg\neg\neg\neg(\knows{\CM}{M}\land\phi)
						\end{array}$\hfill IL
				\item $\LIiPded\neg\neg\neg(\neg(\knows{\CM}{M}\land\phi))\lequiv 
						\neg(\neg(\knows{\CM}{M}\land\phi))$\hfill IL (triple-negation law)
				\item $\LIiPded\neg\neg(\knows{\CM}{M}\land\phi)\lequiv\Iproofdiamond{M}{\phi}{\CM}{}$\hfill definition
				\item $\LIiPded(\Iproofdiamond{M}{(\Iproofdiamond{M}{\phi}{\CM}{})}{\CM}{})\limp
						\Iproofdiamond{M}{\phi}{\CM}{}$\hfill 1--5, IL
				\item $\LIiPded(\Iproofdiamond{M}{\phi}{\CM}{})\limp
						\Iproofdiamond{M}{(\Iproofdiamond{M}{\phi}{\CM}{})}{\CM}{}$\hfill WMM
				\item $\LIiPded(\Iproofdiamond{M}{(\Iproofdiamond{M}{\phi}{\CM}{})}{\CM}{})\lequiv
						\Iproofdiamond{M}{\phi}{\CM}{}$\hfill 6, 7, IL.	
			\end{enumerate}
	For 38, 
		consider MM, ID, and that:
			\begin{enumerate}
				\item $\LIiPded\phi\limp\Iproves{M}{\phi}{\CM}{}$\hfill MM
				\item $\LIiPded(\Iproves{M}{\phi}{\CM}{})\limp\Iproofdiamond{M}{(\Iproves{M}{\phi}{\CM}{})}{\CM}{}$\hfill WMM
				\item $\LIiPded\phi\limp\Iproofdiamond{M}{(\Iproves{M}{\phi}{\CM}{})}{\CM}{}$\hfill 1, 2, IL.
			\end{enumerate}
			and also that:
			\begin{enumerate}
				\item $\LIiPded(\Iproofdiamond{M}{\phi}{\CM}{})\lequiv\neg\neg(\knows{\CM}{M}\land\phi)$\hfill definition
				\item $\LIiPded\knows{\CM}{M}\limp((\Iproves{M}{\phi}{\CM}{})\lequiv\phi)$\hfill ET \emph{bis}
				\item $\LIiPded\neg\neg(\knows{\CM}{M}\land\phi)\lequiv
								\neg\neg(\knows{\CM}{M}\land\Iproves{M}{\phi}{\CM}{})$\hfill 2, IL
				\item $\LIiPded\neg\neg(\knows{\CM}{M}\land\Iproves{M}{\phi}{\CM}{})\lequiv
							\Iproofdiamond{M}{(\Iproves{M}{\phi}{\CM}{})}{\CM}{}$\hfill	definition
				\item $\LIiPded(\Iproofdiamond{M}{\phi}{\CM}{})\lequiv\Iproofdiamond{M}{(\Iproves{M}{\phi}{\CM}{})}{\CM}{}$\hfill  1, 3, 4, IL.
			\end{enumerate}				
	For 39, 
		consider that:
			\begin{enumerate}
				\item $\LIiPded(\Iproves{M}{(\Iproofdiamond{M}{\phi}{\CM}{})}{\CM}{})\limp
								\Iproofdiamond{M}{(\Iproofdiamond{M}{\phi}{\CM}{})}{\CM}{})$\hfill ID
				\item $\LIiPded(\Iproofdiamond{M}{(\Iproofdiamond{M}{\phi}{\CM}{})}{\CM}{})\limp
								\Iproofdiamond{M}{\phi}{\CM}{}$\hfill MI \emph{bis}
				\item $\LIiPded\Iproofdiamond{M}{\phi}{\CM}{}\limp\Iproofdiamond{M}{(\Iproves{M}{\phi}{\CM}{})}{\CM}{}$\hfill NMM
				\item $\LIiPded(\Iproves{M}{(\Iproofdiamond{M}{\phi}{\CM}{})}{\CM}{})\limp\Iproofdiamond{M}{(\Iproves{M}{\phi}{\CM}{})}{\CM}{}$\hfill 1--3, IL.
			\end{enumerate}				
	For 40, consider 
		the instance $\LIiPded\knows{\CM}{M}\limp((\Iproves{M}{(\phi\lor\phi')}{\CM}{})\limp(\phi\lor\phi'))$ of ET \emph{bis}, and 
		that $\LIiPded(\phi\lor\phi')\limp((\Iproves{M}{\phi}{\CM}{})\lor\Iproves{M}{\phi'}{\CM}{})$ by MM.
	Hence $\LIiPded\knows{\CM}{M}\limp((\Iproves{M}{(\phi\lor\phi)}{\CM}{})\limp((\Iproves{M}{\phi}{\CM}{})\lor\Iproves{M}{\phi'}{\CM}{}))$.
	For 41,
		consider 27 and 40.
	For 42,
		consider that:
			\begin{enumerate}
				\item $\LIiPded(\Iproves{M}{(\knows{\CM}{M}\land\phi)}{\CM}{})\limp
							((\Iproves{M}{\knows{\CM}{M}}{\CM}{})\land\Iproves{M}{\phi}{\CM}{})$\hfill pr.\ conj.\ \emph{bis}
				\item $\LIiPded((\Iproves{M}{\knows{\CM}{M}}{\CM}{})\land\Iproves{M}{\phi}{\CM}{})\limp
						\Iproves{M}{\phi}{\CM}{}$\hfill IL
				\item $\LIiPded(\Iproves{M}{(\knows{\CM}{M}\land\phi)}{\CM}{})\limp\Iproves{M}{\phi}{\CM}{}$\hfill 1, 2, IL
				\item $\LIiPded\Iproves{M}{\knows{\CM}{M}}{\CM}{}$\hfill self-knowledge
				\item $\LIiPded(\Iproves{M}{\phi}{\CM}{})\limp\Iproves{M}{\phi}{\CM}{}$\hfill IL
				\item $\LIiPded(\Iproves{M}{\phi}{\CM}{})\limp((\Iproves{M}{\knows{\CM}{M}}{\CM}{})\land\Iproves{M}{\phi}{\CM}{})$\hfill 4, 5, IL
				\item $\LIiPded((\Iproves{M}{\knows{\CM}{M}}{\CM}{})\land\Iproves{M}{\phi}{\CM}{})\limp\Iproves{M}{(\knows{\CM}{M}\land\phi)}{\CM}{}$\hfill pr.\ conj.\ \emph{bis}
				\item $\LIiPded(\Iproves{M}{\phi}{\CM}{})\limp\Iproves{M}{(\knows{\CM}{M}\land\phi)}{\CM}{}$\hfill 6, 7, IL
				\item $\LIiPded(\Iproves{M}{\phi}{\CM}{})\lequiv\Iproves{M}{(\knows{\CM}{M}\land\phi)}{\CM}{}$\hfill 3, 8, IL.
			\end{enumerate}
	For 43,
		consider that :
			\begin{enumerate}
				\item $\LIiPded(\Iproves{M}{\phi}{\CM}{})\lequiv\Iproves{M}{(\knows{\CM}{M}\land\phi)}{\CM}{}$\hfill EI 
				\item $\LIiPded\knows{\CM}{M}\limp((\Iproves{M}{\phi}{\CM}{})\lequiv\phi)$\hfill ET \emph{bis} 
				\item $\LIiPded(\knows{\CM}{M}\land\phi)\lequiv
								(\knows{\CM}{M}\land\Iproves{M}{\phi}{\CM}{})$\hfill 2, IL
				\item $\LIiPded\Iproves{M}{(\knows{\CM}{M}\land\phi)}{\CM}{}\lequiv
								\Iproves{M}{(\knows{\CM}{M}\land\Iproves{M}{\phi}{\CM}{})}{\CM}{}$\hfill 3, R \emph{bis}
				\item $\LIiPded(\Iproves{M}{\phi}{\CM}{})\lequiv\Iproves{M}{(\knows{\CM}{M}\land\Iproves{M}{\phi}{\CM}{})}{\CM}{}$\hfill 1, 4, IL.
			\end{enumerate}
	For 44,
		consider that:
			\begin{enumerate}
				\item $\LIiPded(\Iproofdiamond{M}{(\phi\lor\phi')}{\CM}{})\lequiv
					\neg\neg(\knows{\CM}{M}\land(\phi\lor\phi'))$\hfill definition
				\item $\LIiPded\neg\neg(\knows{\CM}{M}\land(\phi\lor\phi'))\lequiv
						\neg\neg((\knows{\CM}{M}\land\phi)\lor(\knows{\CM}{M}\land\phi'))$\hfill IL
				\item $\LIiPded\begin{array}{@{}l@{}}
							\neg\neg((\knows{\CM}{M}\land\phi)\lor(\knows{\CM}{M}\land\phi'))\lequiv\\
						(\neg\neg(\knows{\CM}{M}\land\phi)\lor\neg\neg(\knows{\CM}{M}\land\phi'))
						\end{array}$\hfill IL
				\item $\LIiPded(\Iproofdiamond{M}{(\phi\lor\phi')}{\CM}{})\lequiv
						(\neg\neg(\knows{\CM}{M}\land\phi)\lor\neg\neg(\knows{\CM}{M}\land\phi'))$\hfill 1--3, IL
				\item $\LIiPded(\Iproofdiamond{M}{\phi}{\CM}{})\lequiv
					\neg\neg(\knows{\CM}{M}\land\phi)$\hfill definition
				\item $\LIiPded(\Iproofdiamond{M}{\phi'}{\CM}{})\lequiv
					\neg\neg(\knows{\CM}{M}\land\phi')$\hfill definition
				\item $\LIiPded(\Iproofdiamond{M}{(\phi\lor\phi')}{\CM}{})\lequiv
			((\Iproofdiamond{M}{\phi}{\CM}{})\lor\Iproofdiamond{M}{\phi'}{\CM}{})$\hfill 4, 5, 6, IL.
			\end{enumerate}
	For 45,
		consider that:  
			\begin{enumerate}
				\item $\LIiPded(\Iproves{M}{(\phi\limp\phi')}{\CM}{})\limp(\knows{\CM}{M}\limp(\phi\limp\phi'))$\hfill ET 
				\item $\LIiPded(\knows{\CM}{M}\limp(\phi\limp\phi'))\lequiv
				((\knows{\CM}{M}\limp\phi)\limp(\knows{\CM}{M}\limp\phi'))$\hfill IL 
				\item $\LIiPded\begin{array}{@{}l@{}}
					((\knows{\CM}{M}\limp\phi)\limp(\knows{\CM}{M}\limp\phi'))\limp\\
						((\knows{\CM}{M}\land(\knows{\CM}{M}\limp\phi))\limp
						 (\knows{\CM}{M}\land(\knows{\CM}{M}\limp\phi')))
					 \end{array}$\hfill IL
				\item $\LIiPded\begin{array}{@{}l@{}}
						((\knows{\CM}{M}\land(\knows{\CM}{M}\limp\phi))\limp
						 (\knows{\CM}{M}\land(\knows{\CM}{M}\limp\phi')))\lequiv\\
						((\knows{\CM}{M}\land\phi)\limp
						 (\knows{\CM}{M}\land\phi')) 
					 \end{array}$\hfill IL
				\item $\LIiPded\begin{array}{@{}l@{}}
						((\knows{\CM}{M}\land\phi)\limp
						 (\knows{\CM}{M}\land\phi'))\limp\\
						(\neg\neg(\knows{\CM}{M}\land\phi)\limp
						 \neg\neg(\knows{\CM}{M}\land\phi'))
					 \end{array}$\hfill IL
				\item $\LIiPded\begin{array}{@{}l@{}}
						(\neg\neg(\knows{\CM}{M}\land\phi)\limp\neg\neg(\knows{\CM}{M}\land\phi'))\lequiv\\
						((\Iproofdiamond{M}{\phi}{\CM}{})\limp\Iproofdiamond{M}{\phi'}{\CM}{})
					 \end{array}$\hfill definition
				\item $\LIiPded(\Iproves{M}{(\phi\limp\phi')}{\CM}{})\limp((\Iproofdiamond{M}{\phi}{\CM}{})\limp\Iproofdiamond{M}{\phi'}{\CM}{})$\hfill 1--6, IL.
			\end{enumerate}
	For 46, 
		consider that:
			\begin{enumerate}
				\item $\LIiPded\phi\limp\Iproofdiamond{M}{\phi}{\CM}{}$\hfill WMM
				\item $\LIiPded(\Iproves{M}{\phi'}{\CM}{})\limp(\knows{\CM}{M}\limp\phi')$\hfill ET
				\item $\LIiPded((\Iproofdiamond{M}{\phi}{\CM}{})\limp\Iproves{M}{\phi'}{\CM}{})\limp(\phi\limp(\knows{\CM}{M}\limp\phi'))$\hfill 1, 2, IL
				\item $\LIiPded(\phi\limp(\knows{\CM}{M}\limp\phi'))\limp\Iproves{M}{(\phi\limp(\knows{\CM}{M}\limp\phi'))}{\CM}{}$\hfill MM
				\item $\LIiPded((\Iproofdiamond{M}{\phi}{\CM}{})\limp\Iproves{M}{\phi'}{\CM}{})\limp\Iproves{M}{(\phi\limp(\knows{\CM}{M}\limp\phi'))}{\CM}{}$\hfill 3, 4, IL
				\item $\LIiPded\begin{array}{@{}l@{}}
						(\Iproves{M}{(\phi\limp(\knows{\CM}{M}\limp\phi'))}{\CM}{})\lequiv\\
						\Iproves{M}{(\knows{\CM}{M}\land(\phi\limp(\knows{\CM}{M}\limp\phi')))}{\CM}{}
						\end{array}$\hfill EI
				\item $\LIiPded\begin{array}{@{}l@{}}
						((\Iproofdiamond{M}{\phi}{\CM}{})\limp\Iproves{M}{\phi'}{\CM}{})\limp\\
						\Iproves{M}{(\knows{\CM}{M}\land(\phi\limp(\knows{\CM}{M}\limp\phi')))}{\CM}{}
						\end{array}$\hfill 5, 6, IL
				\item $\LIiPded(\knows{\CM}{M}\land(\phi\limp(\knows{\CM}{M}\limp\phi')))\limp
						(\phi\limp\phi')$\hfill IL
				\item $\LIiPded\Iproves{M}{(\knows{\CM}{M}\land(\phi\limp(\knows{\CM}{M}\limp\phi')))}{\CM}{}\limp\Iproves{M}{(\phi\limp\phi')}{\CM}{}$\hfill 8, R
				\item $\LIiPded((\Iproofdiamond{M}{\phi}{\CM}{})\limp\Iproves{M}{\phi'}{\CM}{})\limp
						\Iproves{M}{(\phi\limp\phi')}{\CM}{}$\hfill 7, 9, IL.
			\end{enumerate}
	For 47,
		inspect CF and PS5.
	For 48, 
		consider that:
			\begin{enumerate}
				\item $\LIiPded(\neg\neg\phi\limp\phi')\limp
						((\Iproofdiamond{M}{\phi}{\CM}{})\limp\Iproves{M}{\phi'}{\CM}{})$\hfill CF
				\item $\LIiPded(\neg\neg\phi'\limp\phi)\limp
						((\Iproofdiamond{M}{\phi'}{\CM}{})\limp\Iproves{M}{\phi}{\CM}{})$\hfill CF		
				\item $\LIiPded(\Iproofdiamond{M}{(\phi\lor\phi')}{\CM}{})\lequiv
						((\Iproofdiamond{M}{\phi}{\CM}{})\lor\Iproofdiamond{M}{\phi'}{\CM}{})$\hfill PS4
				\item $\LIiPded(\Iproves{M}{(\phi\lor\phi')}{\CM}{})\limp 
								\Iproofdiamond{M}{(\phi\lor\phi')}{\CM}{}$\hfill ID
				\item\begin{tabular}{@{}l@{}l@{}}
					$\LIiPded$&$((\neg\neg\phi\limp\phi')\land(\neg\neg\phi'\limp\phi))\limp$\\  
						&$((\Iproves{M}{(\phi\lor\phi')}{\CM}{})\limp((\Iproves{M}{\phi}{\CM}{})\lor\Iproves{M}{\phi'}{\CM}{}))$
				\end{tabular}\hfill 1--4, IL.
			\end{enumerate}
	For 49,
		consider that 
			$\LIiPded(\Iproofdiamond{M}{\phi}{\CM}{})\limp
						\Iproves{M}{(\Iproofdiamond{M}{\phi}{\CM}{})}{\CM}{}$ by MM; thus 
			$\LIiPded(\Iproves{M}{\phi}{\CM}{})\limp
						\Iproves{M}{(\Iproofdiamond{M}{\phi}{\CM}{})}{\CM}{}$ by ID; and thus
			$\LIiPded\phi\limp\Iproves{M}{(\Iproofdiamond{M}{\phi}{\CM}{})}{\CM}{}$ by MM.			
	For 50,
		consider that:
			\begin{enumerate}
				\item $\LIiPded(\Iproves{M}{\phi}{\CM}{})\limp\Iproofdiamond{M}{\phi}{\CM}{}$\hfill ID
				\item $\LIiPded\Iproves{M}{((\Iproves{M}{\phi}{\CM}{})\limp\Iproofdiamond{M}{\phi}{\CM}{})}{\CM}{}$\hfill 1, N
				\item $\LIiPded(\Iproofdiamond{M}{(\Iproves{M}{\phi}{\CM}{})}{\CM}{})\limp
								\Iproofdiamond{M}{(\Iproofdiamond{M}{\phi}{\CM}{})}{\CM}{}$\hfill 2, PS2
				\item $\LIiPded(\Iproofdiamond{M}{(\Iproofdiamond{M}{\phi}{\CM}{})}{\CM}{})\limp
								\Iproofdiamond{M}{\phi}{\CM}{}$\hfill MI \emph{bis}
				\item $\LIiPded(\Iproofdiamond{M}{\phi}{\CM}{})\limp 
								\Iproves{M}{(\Iproofdiamond{M}{\phi}{\CM}{})}{\CM}{}$\hfill MM
				\item $\LIiPded(\Iproofdiamond{M}{(\Iproves{M}{\phi}{\CM}{})}{\CM}{})\limp
						\Iproves{M}{(\Iproofdiamond{M}{\phi}{\CM}{})}{\CM}{}$\hfill 3--5, IL.
			\end{enumerate}
	For 51, combine MS and MS \emph{bis}.
	For 52, consider the in fact stronger-than-necessary proof of 38.
	For 53, consider that:
		\begin{enumerate}
				\item $\LIiPded(\Iproves{M}{(\Iproofdiamond{M}{\phi}{\CM}{})}{\CM}{})\limp
								\Iproofdiamond{M}{(\Iproofdiamond{M}{\phi}{\CM}{})}{\CM}{}$\hfill ID
				\item $\LIiPded(\Iproofdiamond{M}{(\Iproofdiamond{M}{\phi}{\CM}{})}{\CM}{})\lequiv 
								\Iproofdiamond{M}{\phi}{\CM}{}$\hfill MI \emph{bis}
				\item $\LIiPded(\Iproves{M}{(\Iproofdiamond{M}{\phi}{\CM}{})}{\CM}{})\limp
								\Iproofdiamond{M}{\phi}{\CM}{}$\hfill 1, 2, IL
				\item $\LIiPded(\Iproofdiamond{M}{\phi}{\CM}{})\limp
								\Iproves{M}{(\Iproofdiamond{M}{\phi}{\CM}{})}{\CM}{}$\hfill MM
				\item $\LIiPded(\Iproves{M}{(\Iproofdiamond{M}{\phi}{\CM}{})}{\CM}{})\lequiv
								\Iproofdiamond{M}{\phi}{\CM}{}$\hfill 3, 4 IL.
			\end{enumerate}
	For 54, consider that:
		\begin{enumerate}	
			\item $\LIiPded(\Iproves{M}{\neg\neg\phi}{\CM}{})\limp\Iproves{M}{(\knows{\CM}{M}\land\neg\neg\phi)}{\CM}{}$\hfill EI
			\item $\LIiPded(\knows{\CM}{M}\land\neg\neg\phi)\limp\neg\neg(\knows{\CM}{M}\land\phi)$\hfill IL
			\item $\LIiPded(\Iproves{M}{(\knows{\CM}{M}\land\neg\neg\phi)}{\CM}{})\limp
								\Iproves{M}{\neg\neg(\knows{\CM}{M}\land\phi)}{\CM}{}$\hfill 2, R
			\item $\LIiPded(\Iproves{M}{\neg\neg\phi}{\CM}{})\limp\Iproves{M}{\neg\neg(\knows{\CM}{M}\land\phi)}{\CM}{}$\hfill 1, 3, IL
			\item $\LIiPded(\Iproves{M}{\neg\neg(\knows{\CM}{M}\land\phi)}{\CM}{})\lequiv
							\Iproves{M}{(\Iproofdiamond{M}{\phi}{\CM}{})}{\CM}{}$\hfill definition
			\item $\LIiPded(\Iproves{M}{\neg\neg\phi}{\CM}{})\limp\Iproves{M}{(\Iproofdiamond{M}{\phi}{\CM}{})}{\CM}{}$\hfill 4, 5, IL
			\item $\LIiPded(\Iproves{M}{(\Iproofdiamond{M}{\phi}{\CM}{})}{\CM}{})\lequiv
							\Iproofdiamond{M}{\phi}{\CM}{}$\hfill MMI \emph{bis}
			\item $\LIiPded(\Iproves{M}{\neg\neg\phi}{\CM}{})\limp
							\Iproofdiamond{M}{\phi}{\CM}{}$\hfill 6, 7, IL.
		\end{enumerate}
	For 55, consider that:
		\begin{enumerate}
			\item $\LIiPded(\Iproves{M}{\neg\neg\phi}{\CM}{})\limp
							\Iproofdiamond{M}{\phi}{\CM}{}$\hfill DNA
			\item $\LIiPded(\Iproofdiamond{M}{\phi}{\CM}{})\limp\neg\neg\phi$\hfill EWDN
			\item $\LIiPded\phi\limp\Iproves{M}{\phi}{\CM}{}$\hfill MM
			\item $\LIiPded(\Iproofdiamond{M}{\phi}{\CM}{})\limp(\neg\neg\phi\land(\phi\limp\Iproves{M}{\phi}{\CM}{}))$\hfill 2, 3, IL
			\item $\LIiPded(\neg\neg\phi\land(\phi\limp\Iproves{M}{\phi}{\CM}{}))\limp
							\neg\neg(\Iproves{M}{\phi}{\CM}{})$\hfill IL
			\item $\LIiPded(\Iproves{M}{\neg\neg\phi}{\CM}{})\limp
						\neg\neg(\Iproves{M}{\phi}{\CM}{})$\hfill 1, 4, 5, IL.
		\end{enumerate}	
	For 56, consider that:
		\begin{enumerate}
			\item $\LIiPded\neg(\Iproofdiamond{M}{\phi}{\CM}{})\lequiv
							\neg\neg\neg(\knows{\CM}{M}\land\phi)$\hfill definition
			\item $\LIiPded\neg\neg\neg(\knows{\CM}{M}\land\phi)\lequiv\neg(\knows{\CM}{M}\land\phi)$\hfill IL (triple-negation law) 
			\item $\LIiPded\neg(\knows{\CM}{M}\land\phi)\lequiv(\knows{\CM}{M}\limp\neg\phi)$\hfill IL
			\item $\LIiPded\neg\neg(\knows{\CM}{M})$\hfill CMMC
			\item $\LIiPded(\knows{\CM}{M}\limp\neg\phi)\lequiv
							(\neg\neg(\knows{\CM}{M})\land (\knows{\CM}{M}\limp\neg\phi))$\hfill 4, IL
			\item $\LIiPded(\neg\neg(\knows{\CM}{M})\land (\knows{\CM}{M}\limp\neg\phi))\limp
							\neg\neg(\knows{\CM}{M}\land\neg\phi)$\hfill IL
			\item $\LIiPded\neg\neg(\knows{\CM}{M}\land\neg\phi)\lequiv\Iproofdiamond{M}{\neg\phi}{\CM}{}$\hfill definition
			\item $\LIiPded\neg(\Iproofdiamond{M}{\phi}{\CM}{})\limp\Iproofdiamond{M}{\neg\phi}{\CM}{}$\hfill 1--7, IL.
		\end{enumerate}	
	For 57, combine FNDETU and WNC.
	For 58, consider that:
		\begin{enumerate}
			\item $\LIiPded(\Iproofdiamond{M}{\phi}{\CM}{})\limp\neg\neg\phi$\hfill EWDN
			\item $\LIiPded\phi\limp\Iproves{M}{\phi}{\CM}{}$\hfill MM
			\item $\LIiPded(\neg\neg\phi\land(\phi\limp\Iproves{M}{\phi}{\CM}{}))\limp\neg\neg(\Iproves{M}{\phi}{\CM}{})$\hfill IL
			\item $\LIiPded\neg\neg\phi\limp\neg\neg(\Iproves{M}{\phi}{\CM}{})$\hfill 2, 3, IL
			\item $\LIiPded(\Iproofdiamond{M}{\phi}{\CM}{})\limp\neg\neg(\Iproves{M}{\phi}{\CM}{})$\hfill 1, 4, IL
			\item $\LIiPded(\Iproves{M}{\phi}{\CM}{})\limp\Iproofdiamond{M}{\phi}{\CM}{}$\hfill ID
			\item $\LIiPded(\neg\neg(\Iproves{M}{\phi}{\CM}{})\land((\Iproves{M}{\phi}{\CM}{})\limp\Iproofdiamond{M}{\phi}{\CM}{}))\limp\neg\neg(\Iproofdiamond{M}{\phi}{\CM}{})$\hfill IL
			\item $\LIiPded\neg\neg(\Iproves{M}{\phi}{\CM}{})\limp\neg\neg(\Iproofdiamond{M}{\phi}{\CM}{})$\hfill 6, 7, IL
			\item $\LIiPded\neg\neg(\Iproofdiamond{M}{\phi}{\CM}{})\lequiv\neg\neg\neg\neg(\knows{\CM}{M}\land\phi)$\hfill definition
			\item $\LIiPded\neg\neg\neg\neg(\knows{\CM}{M}\land\phi)\lequiv
							\neg\neg(\knows{\CM}{M}\land\phi)$\hfill triple-negation law
			\item $\LIiPded\neg\neg(\knows{\CM}{M}\land\phi)\lequiv\Iproofdiamond{M}{\phi}{\CM}{}$\hfill definition
			\item $\LIiPded\neg\neg(\Iproves{M}{\phi}{\CM}{})\limp\Iproofdiamond{M}{\phi}{\CM}{}$\hfill 8--11, IL
			\item $\LIiPded(\Iproofdiamond{M}{\phi}{\CM}{})\lequiv\neg\neg(\Iproves{M}{\phi}{\CM}{})$\hfill 5, 12, IL.
		\end{enumerate}	
	For 59, consider that:
		\begin{enumerate}
			\item $\LIiPded(\Iproves{M}{(\phi\limp\phi')}{\CM}{})\limp
						((\Iproves{M}{\phi}{\CM}{})\limp\Iproves{M}{\phi'}{\CM}{})$\hfill K
			\item $\LIiPded\begin{array}{@{}l@{}}
						((\Iproves{M}{\phi}{\CM}{})\land\Iproofdiamond{M}{(\phi\limp\phi')}{\CM}{})\limp\\
								(((\Iproves{M}{(\phi\limp\phi')}{\CM}{})\limp\Iproves{M}{\phi'}{\CM}{})\land 
									\Iproofdiamond{M}{(\phi\limp\phi')}{\CM}{})
							\end{array}$\hfill 1, IL
			\item $\LIiPded(\Iproofdiamond{M}{(\phi\limp\phi')}{\CM}{})\lequiv\neg\neg(\Iproves{M}{(\phi\limp\phi')}{\CM}{})$\hfill MDN
			\item $\LIiPded\begin{array}{@{}l@{}}
								(((\Iproves{M}{(\phi\limp\phi')}{\CM}{})\limp\Iproves{M}{\phi'}{\CM}{})\land 
									\Iproofdiamond{M}{(\phi\limp\phi')}{\CM}{})\lequiv\\ 
								(((\Iproves{M}{(\phi\limp\phi')}{\CM}{})\limp\Iproves{M}{\phi'}{\CM}{})\land 
									\neg\neg(\Iproves{M}{(\phi\limp\phi')}{\CM}{}))
							\end{array}$\hfill 3, IL
			\item $\LIiPded\begin{array}{@{}l@{}}
								(((\Iproves{M}{(\phi\limp\phi')}{\CM}{})\limp\Iproves{M}{\phi'}{\CM}{})\land 
									\neg\neg(\Iproves{M}{(\phi\limp\phi')}{\CM}{}))\limp\\
									\neg\neg(\Iproves{M}{\phi'}{\CM}{})
							\end{array}$\hfill IL
			\item $\LIiPded\neg\neg(\Iproves{M}{\phi'}{\CM}{})\lequiv\Iproofdiamond{M}{\phi'}{\CM}{}$\hfill MDN
			\item $\LIiPded((\Iproves{M}{\phi}{\CM}{})\land\Iproofdiamond{M}{(\phi\limp\phi')}{\CM}{})\limp\Iproofdiamond{M}{\phi'}{\CM}{}$\hfill 2, 4, 5, 6, IL.
		\end{enumerate}	
	For 60, consider that:
		\begin{enumerate}
			\item $\LIiPded(\Iproofdiamond{M}{(\phi\land\neg\phi)}{\CM}{})\limp\neg\neg(\phi\land\neg\phi)$\hfill EWDN
			\item $\LIiPded\neg\neg(\phi\land\neg\phi)\limp(\neg(\phi\land\neg\phi)\limp(\phi\land\neg\phi))$\hfill IL (\emph{reductio ad absurdum})
			\item $\LIiPded\false\limp\false$\hfill IL
			\item $\LIiPded(\false\limp\false)\lequiv\neg\false$\hfill IL
			\item $\LIiPded\neg\false\lequiv\neg(\phi\land\neg\phi)$\hfill IL
			\item $\LIiPded\neg(\phi\land\neg\phi)$\hfill 3, 4, 5, IL
			\item $\LIiPded\neg\neg(\phi\land\neg\phi)\limp(\phi\land\neg\phi)$\hfill 2, 6, IL
			\item $\LIiPded(\Iproofdiamond{M}{(\phi\land\neg\phi)}{\CM}{})\limp(\phi\land\neg\phi)$\hfill 1, 7, IL.
		\end{enumerate}

\subsection{Finite-model property}\label{appendix:FiniteModelProperty}
			\begin{itemize}
				\item $\sqsubseteq^{\mathrm{min},\Gamma}$ inherits 
						the reflexivity, transitivity, and anti-symmetry from $\sqsubseteq$ as
							can be seen by inspecting the definition of $\sqsubseteq^{\mathrm{min},\Gamma}$;
				\item ${\access{M}{\CM}{\mathrm{min},\Gamma}}$ inherits 
					seriality from ${\access{M}{\CM}{}}$, 
						as can be seen by inspecting the definition of ${\access{M}{a}{\mathrm{min},\Gamma}}$;
				\item for conditional reflexivity of ${\access{M}{\CM}{\mathrm{min},\Gamma}}$, suppose that
						$M\in\clo{\CM}{[s]_{\sim_{\Gamma}}}(\emptyset)$.
						Thus consecutively: 
							$[s]_{\sim_{\Gamma}}\in\mathcal{V}_{\Gamma}(\knows{\CM}{M})$ by definition,
							$s\in\mathcal{V}(\knows{\CM}{M})$ by definition,
							$M\in\clo{\CM}{s}(\emptyset)$ by definition,
							$s\access{M}{\CM}{}s$ by the conditional reflexivity of $\access{M}{\CM}{}$, and finally 
							$[s]_{\sim_{\Gamma}}\access{M}{\CM}{\mathrm{min},\Gamma}[s]_{\sim_{\Gamma}}$ by 
								definition;
				\item for the epistemic-image property of ${\access{M}{\CM}{\mathrm{min},\Gamma}}$, suppose that
						$[s]_{\sim_{\Gamma}}\access{M}{\CM}{\mathrm{min},\Gamma}[s']_{\sim_{\Gamma}}$.
						Thus consecutively: 
							$s\access{M}{\CM}{}s'$ by definition,
							$M\in\clo{\CM}{s'}(\emptyset)$ by the epistemic-image property of $\access{M}{\CM}{}$,
							$s'\in\mathcal{V}(\knows{\CM}{M})$ by definition,
							$[s']_{\sim_{\Gamma}}\in\mathcal{V}_{\Gamma}(\knows{\CM}{M})$ by definition, and finally
							$M\in\clo{\CM}{[s']_{\sim_{\Gamma}}}(\emptyset)$ by definition.
				\item For the MIAR-inclusion property of ${\access{M}{\CM}{\mathrm{min},\Gamma}}$, suppose that:
						\begin{itemize}
							\item 
							$[s]_{\sim_{\Gamma}}\access{\CM}{\CM}{\mathrm{min},\Gamma}[s']_{\sim_{\Gamma}}$.
						Thus consecutively:
							$s\access{\CM}{\CM}{}s'$ by definition,
							$s\sqsubseteq s'$ by MIAR-inclusion, and 
							$[s]_{\sim_{\Gamma}}\sqsubseteq^{\mathrm{min},\Gamma}[s']_{\sim_{\Gamma}}$ by definition. Proceed similarly for the converse.
						
							\item $[s]_{\sim_{\Gamma}}\access{M}{\CM}{\mathrm{min},\Gamma}[s']_{\sim_{\Gamma}}$.
						Thus consecutively:
							$s\access{M}{\CM}{}s'$ by definition,						
							$s\access{\CM}{\CM}{}s'$ by MIAR-inclusion, and 
							$[s]_{\sim_{\Gamma}}\access{\CM}{\CM}{\mathrm{min},\Gamma}[s']_{\sim_{\Gamma}}$ by definition.
					\end{itemize}
				\item For the special transitivity of ${\access{M}{\CM}{\mathrm{min},\Gamma}}$,
					suppose that 
						$$[s]_{\sim_{\Gamma}}\mathrel{(\access{\CM}{\CM}{\mathrm{min},\Gamma}\circ
														\access{M}{\CM}{\mathrm{min},\Gamma})}[s']_{\sim_{\Gamma}}\,.$$
					That is,
						there is $[s'']\in\states/_{\sim_{\Gamma}}$ such that 
							\begin{itemize}
								\item $[s]_{\sim_{\Gamma}}\access{\CM}{\CM}{\mathrm{min},\Gamma}[s'']_{\sim_{\Gamma}}$ and
								\item $[s'']_{\sim_{\Gamma}}\access{M}{\CM}{\mathrm{min},\Gamma}[s']_{\sim_{\Gamma}}$\,.
							\end{itemize}
					Thus consecutively:
						$s\access{\CM}{\CM}{}s''$ and $s''\access{M}{\CM}{}s'$ by definition,
						$s\access{M}{\CM}{}s''$ by MIAR-inclusion, and 
						$[s]_{\sim_{\Gamma}}\access{M}{\CM}{\mathrm{min},\Gamma}[s']_{\sim_{\Gamma}}$ by definition.				
				\item For the proof monotonicity of ${\access{M}{\CM}{\mathrm{min},\Gamma}}$, 
						suppose that $M\sqsubseteq_{\CM}M'$.
						Further suppose that 
							$[s]_{\sim_{\Gamma}}\access{M}{\CM}{\mathrm{min},\Gamma}[s']_{\sim_{\Gamma}}$.
						Thus consecutively:
							$s\access{M}{\CM}{}s'$ by definition,
							$s\access{M'}{\CM}{}s'$ by proof monotonicity, and 
							$[s]_{\sim_{\Gamma}}\access{M'}{\CM}{\mathrm{min},\Gamma}[s']_{\sim_{\Gamma}}$ by definition.					
			\end{itemize}

\end{document}